\documentclass[10pt]{article}
\usepackage{amsmath, epsfig}
\usepackage{latexsym}
\usepackage{amssymb}
\usepackage{color}
\usepackage{mathrsfs}

\def\l{\lambda}
\def\De{\Delta}

\makeatother

\newtheorem{thm}{Theorem}[section]
\newcommand{\n}{\noindent}
\newtheorem{conj}[thm]{Conjecture}
\newtheorem{lemma}[thm]{Lemma}
\newtheorem{question}[thm]{Question}

\newtheorem{defn}[thm]{Definition}

\newtheorem{notn}[thm]{Notation}
\newenvironment{proof}{{\bf Proof}.}{\rule{3mm}{3mm}}

\usepackage{color}

\definecolor{Thistle}{rgb}{0.847,0.749,0.847}
\definecolor{Khaki}{rgb}{0.941,0.902,0.549}
\definecolor{Orchid}{rgb}{0.855,0.439,0.839}
\definecolor{MediumOrchid}{rgb}{0.729,0.333,0.827}

\definecolor{brown}{rgb}{0.8,0.5,0}
\definecolor{LightBrown}{rgb}{0.8,0.2,0.4}

\definecolor{DarkGray}{rgb}{0.78,0.78,0.78}
\definecolor{DarkMidGray}{rgb}{0.81,0.81,0.81}
\definecolor{MidGray}{rgb}{0.85,0.85,0.85}
\definecolor{LightGray}{rgb}{0.88,0.88,0.88}
\definecolor{VeryLightGray}{rgb}{0.96,0.96,0.96}

\definecolor{GrayA}{rgb}{0.7,0.7,0.7}
\definecolor{GrayB}{rgb}{0.78,0.78,0.78}
\definecolor{GrayC}{rgb}{0.80,0.80,0.80}
\definecolor{GrayD}{rgb}{0.82,0.82,0.82}
\definecolor{GrayE}{rgb}{0.84,0.84,0.84}
\definecolor{GrayF}{rgb}{0.86,0.86,0.86}
\definecolor{GrayG}{rgb}{0.88,0.88,0.88}
\definecolor{GrayH}{rgb}{0.90,0.90,0.90}
\definecolor{GrayI}{rgb}{0.92,0.92,0.92}
\definecolor{GrayJ}{rgb}{0.94,0.94,0.94}

\definecolor{VeryLightBlue}{rgb}{0.9,0.9,1}
\definecolor{LightBlue}{rgb}{0.8,0.8,1}
\definecolor{MidBlue}{rgb}{0.5,0.5,1}
\definecolor{DarkBlue}{rgb}{0,0,0.6}

\definecolor{Gold}{rgb}{1,0.843,0}

\definecolor{LightGreen}{rgb}{0.88,1,0.88}
\definecolor{MidGreen}{rgb}{0.6,1,0.6}
\definecolor{DarkGreen}{rgb}{0,0.6,0}

\definecolor{VeryLightYellow}{rgb}{1,1,0.9}
\definecolor{LightYellow}{rgb}{1,1,0.6}
\definecolor{MidYellow}{rgb}{1,1,0.5}
\definecolor{DarkYellow}{rgb}{1,1,0.2}

\definecolor{VeryLightRed}{rgb}{1,0.9,0.9}
\definecolor{LightRed}{rgb}{1,0.8,0.8}
\definecolor{MidRed}{rgb}{1,0.55,0.55}

\newcommand{\BLUE}[1]{\textcolor{blue}{#1}}

\newcommand{\sm}[1]{{\BLUE{\bf [Sanming: #1]}}}

\long\def\delete#1{}

\begin{document}

\title{\Large \bf Labeling outerplanar graphs with maximum degree three}

\author{
Xiangwen Li$^{1}$~\thanks{Supported by the
Natural Science Foundation of China (11171129), email: xwli68@yahoo.cn; xwli2808@yahoo.com}\;\, and\;\, Sanming Zhou$^2$~\thanks{Supported by a Future Fellowship (FT110100629) and a Discovery Project Grant (DP120101081) of the Australian Research Council, email: sanming@unimelb.edu.au}\\
{\small $^1$Department of Mathematics}\\
{\small Huazhong Normal University, Wuhan 430079, China}\\
{\small $^2$Department of Mathematics and Statistics}\\
{\small The University of Melbourne, Parkville, VIC 3010, Australia}}

\date{April 19, 2012}

\maketitle

\begin{abstract}
An $L(2, 1)$-labeling of a graph $G$ is an assignment of a nonnegative integer to each vertex of $G$ such that adjacent vertices receive integers that differ by at least two and vertices
at distance two receive distinct integers. The span of such a labeling is the difference between the largest and smallest integers used. The $\lambda$-number of $G$, denoted by $\lambda(G)$, is the minimum span over all $L(2, 1)$-labelings of $G$.
Bodlaender {\it et al.} conjectured that if $G$ is an outerplanar graph of maximum degree
$\Delta$, then $\lambda(G)\leq \Delta+2$. Calamoneri and Petreschi proved that this conjecture is true when $\Delta \geq 8$ but false when $\Delta=3$. Meanwhile, they proved that $\lambda(G)\leq \Delta+5$ for any outerplanar graph $G$ with $\Delta=3$ and asked whether or not this bound is sharp. In this paper we answer this question by proving that $\lambda(G)\leq \Delta + 3$ for every outerplanar graph with maximum degree $\Delta=3$. We also show that this bound $\Delta + 3$ can be achieved by infinitely many outerplanar graphs with $\Delta=3$.
\vspace{2mm}

 \vspace{4mm}

\end{abstract}

\section{Introduction}

In the channel assignment problem \cite{Hale} one wishes to assign channels to
transmitters in a radio communication network such that the bandwidth used
is minimized whilst interference is avoided as much as possible.
Various constraints have been suggested to put on channel
separations between pairs of transmitters within certain distances, leading to several important
optimal labeling problems which are generalizations of the ordinary graph
coloring problem. Among them the $L(2,1)$-labeling problem \cite{GY} has received most
attention in the past two decades.

Given integers $p \ge q \ge 1$, an {\it $L(p, q)$-labeling} of a graph $G=(V(G), E(G))$ is a mapping $f$ from $V(G)$ to the set of nonnegative integers such that $|f(u)-f(v)|
\geq p$ if $u$ and $v$ are adjacent in $G$, and
$|f(u)-f(v)|\geq q$ if $u$ and $v$ are distance two apart in
$G$. The integers used by $f$ are called the {\it labels}, and the
{\it span} of $f$ is the difference between the largest and smallest labels used by $f$.
The {\it $\lambda_{p,q}$-number} of $G$, $\lambda_{p,q}(G)$, is the minimum
span over all $L(p, q)$-labelings of $G$. We may assume without loss of generality that the smallest label used is 0, so that $\lambda_{p,q}(G)$ is equal to the minimum value among the largest labels used by $L(p, q)$-labelings of $G$.
In particular, $\l(G)= \l_{2,1}(G)$ is called the {\em $\l$-number} of $G$.
For a nonnegative integer $k$, a {\it $k$-$L(2, 1)$-labeling} is
an $L(2, 1)$-labeling with maximum label at most $k$. Thus
$\lambda (G)$ is the minimum $k$ such that $G$ admits a $k$-$L(2, 1)$-labeling.

The $L(p, q)$-labelling problem has received extensive attention over many years especially in the case when $(p,q)=(2,1)$ (see \cite{Ca} for a survey). Griggs and Yeh \cite{GY} conjectured that $\lambda(G) \le \Delta^2$ for any graph $G$ with maximum degree $\Delta \ge 2$. This has been confirmed for several classes of graphs, including chordal graphs \cite{Sakai}, generalized Petersen graphs \cite{GM}, Hamiltonian graphs with $\Delta \le 3$ \cite{Kang}, etc. Improving earlier results, Goncalves \cite{G} proved that $\lambda(G) \le \Delta^2+\Delta-2$ for any graph $G$ with $\Delta \ge 2$. Recently, Havet, Reed and Sereni \cite{HRS} proved that for any $p \ge 1$ there exists a constant $\Delta(p)$ such that every graph with maximum degree $\Delta \ge \Delta(p)$ has an $L(p, 1)$-labelling with span at most $\Delta^2$. In particular, this implies that the Griggs-Yeh conjecture is true for any graph with sufficiently large $\Delta$.

The $\lambda$-number of a graph relies not only on its maximum degree but also on its structure. It is thus important to understand this invariant for different graph classes. In particular, for the class of planar graphs, Molloy and Salavatipour \cite{Molloy} proved that $\l_{p,q}(G) \le q \lceil 5\De/3 \rceil + 18p + 77q - 18$, which yields $\l(G) \le \lceil 5\De/3 \rceil + 77$, and Bella {\it et al.} \cite{BKMQ} proved that the Griggs-Yeh conjecture is true for planar graphs with $\De \ne 3$. (See \cite{BKMQ} for a brief survey on the $L(p,q)$-labeling problem for planar graphs.) This indicates that the class of planar graphs with $\De = 3$ may require a special treatment. In \cite{BO}, Bodlaender {\it et al.}  proposed the following conjecture.

\begin{conj}
\label{conj1} For any outerplanar graph $G$ of maximum degree $\Delta$, $\lambda(G)\leq
\Delta+2$.
\end{conj}

Bodlaender {\it et al.} \cite{BO} themselves proved that
$\lambda(G)\leq \Delta+8$ for any outerplanar graph $G$ and $\lambda(G)\leq\Delta+6$ for any triangulated outerplanar graph. Calamoneri and Petreschi~\cite{CR} proved that Conjecture~\ref{conj1} is true for any outerplanar graph with $\Delta\geq 8$. At the same time, as a corollary of their results on circular distance two labelings, Liu and Zhu~\cite{ZHU} proved that this conjecture is true for outerplanar graphs with $\Delta\geq 15$. In \cite{CR}, Calamoneri and Petreschi also proved that $\lambda\leq\Delta+1$ for any triangulated outerplanar graph with $\Delta\geq 8$. Meanwhile, they gave \cite{CR} a couterexample to show that Conjecture~\ref{conj1} is false when $\Delta=3$, indicating again that the case of maximum degree three may require a special treatment. Moreover, for any outerplanar graph $G$ with $\Delta=3$, they proved \cite{CR} that $\lambda(G)\leq \Delta+5$ and $\lambda(G)\leq \Delta+4$ if in addition $G$ is triangle-free. Motivated by these results they asked \cite{CR} the following question.
\begin{question}
Is the upper bound $\lambda(G)\leq \Delta+5$ tight for outerplanar graphs with $\Delta=3$?
\end{question}

Although empirical results \cite{BH} suggest that $\Delta+5$ may be improved for some sample outerplanar graphs with $\Delta=3$, to our best knowledge the question above is still open, and no sharp upper bound on $\lambda(G)$ is known for all outerplanar graphs with $\Delta=3$.  

In this paper we answer the question above by proving that $\lambda(G)\leq \Delta + 3 = 6$  for every outerplanar graph with $\Delta=3$. Moreover, we prove that this bound can be achieved by infinitely many outerplanar graphs with $\Delta=3$. This extends the single outerplanar graph with $\Delta=3$ and $\lambda = 6$ given in \cite{CR} to an infinite family of such extremal graphs. A typical member in this family is the graph $G(l)$ such that $l \ge 4$ is not a multiple of $3$, where $G(l)$, depicted in Figure \ref{fig:gl}, is defined as follows. (Note that $G(4)$ is the graph in \cite[Fig. 8]{CR}.)  

\begin{defn}
{\em Let $l \ge 3$ be an integer. Define $G(l)$ to be the graph with vertex set $\{u, v, x_1, x_2, \ldots, x_l, y_1, \\ y_2, \ldots, y_l\}$ and edge set $\{x_ix_{i+1}, y_i y_{i+1}: 1\leq i\leq l-1\} \cup\{ux_1,uy_1,vx_l, vy_l\} \cup \{x_iy_i: 1\leq i\leq l\}$.} 
\end{defn}

\begin{figure}[ht]
\centering
\includegraphics*[height=10cm]{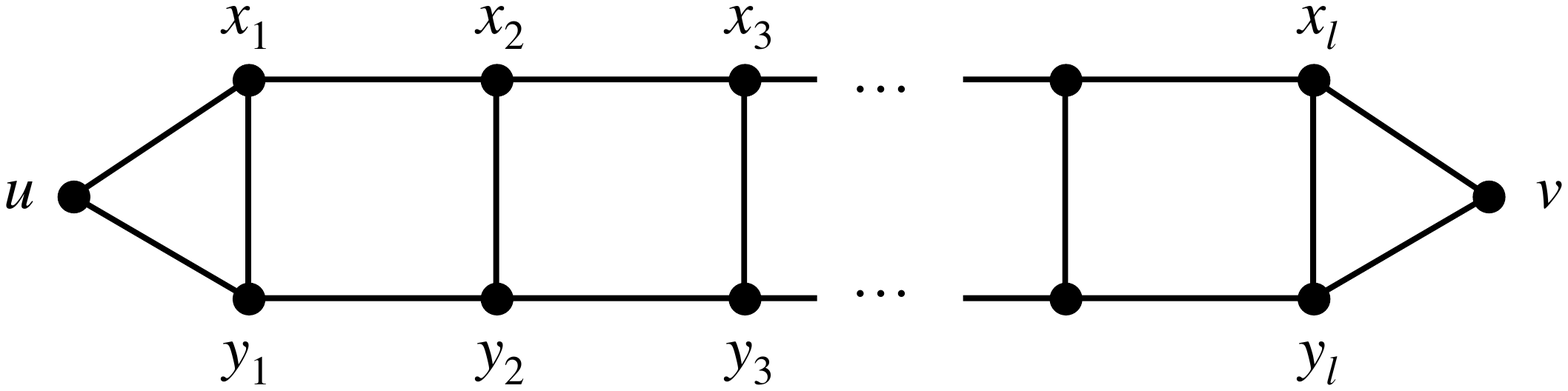}
\vspace{-1.0cm}
\caption{\small The outerplanar graph $G(l)$.} 
\label{fig:gl}
\end{figure} 


The main result of this paper is as follows.

\begin{thm}
\label{thm1}
For any outerplanar graph $G$ with maximum degree $\Delta=3$, we have $\lambda(G)\leq 6$. Moreover, for any integer $l \ge 4$ which is not a multiple of $3$, the outerplanar graph $G(l)$ as shown in Figure \ref{fig:gl} satisfies $\lambda(G(l)) = 6$.
\end{thm}

It is clear that, for any graph $G$, $\l_{1,1}(G)+1$ is equal to the chromatic number $\chi(G^2)$ of the square of $G$. Wegner \cite{W} conjectured that, for any planar graph $G$, $\chi(G^2)$ is bounded from above by 7 if $\De=3$, by $\De + 5$ if $4 \le \De \le 7$, and by $\lfloor 3\De/2 \rfloor + 1$ if $\De \ge 8$. The aforementioned result of Molloy and Salavatipour \cite{Molloy}, which can be restated as $\chi(G^2) \le \lceil 5\De/3 \rceil + 78$, is the best known result on this conjecture for general $\De$. Thomassen \cite{T} proved that Wegner's conjecture is true for planar graphs with $\De = 3$. Since $\chi(G^2) - 1 = \l_{1,1}(G) \le \l(G)$, Theorem \ref{thm1} can be viewed as a refinement of Thomassen's result for outerplanar graphs (but not general planar graphs) with $\De=3$.

In order to prove Theorem \ref{thm1}, we will introduce two extension techniques in Sections \ref{sec:ext1} and \ref{sec:ext2}, respectively. (See Theorems \ref{le43} and \ref{le44}.) Developed from graph coloring theory, they will play a key role in our proof (in Section \ref{sec:ubound}) of the upper bound $\lambda(G)\leq 6$. In Section \ref{sec:lower} we will prove $\lambda(G(l)) \ge 6$, which together with the upper bound implies the second statement in Theorem \ref{thm1}.

We follow \cite{BM} for terminology and notation on graphs. All graphs considered are simple and undirected. The neighborhood of $v$ in a graph $G$ is denoted by $N(v)$, and the set of vertices at distance two from $v$ in $G$ is denoted by $N_2(v)$.
A graph is {\it outerplanar} if it can be embedded in the plane in
such a way that all the vertices lie on the boundary of the same face called the {\it outer
face}. When an outerplanar graph $G$ is drawn in this way in a plane, we call it an
{\it outer plane graph}. The boundary of a face $F$ of an outer plane graph, denoted by $\partial F$,  can be regarded as the subgraph induced by its vertices. If $F$ is circular with vertices $x_1, x_2, \ldots, x_k$ in order, then $\partial F=x_1x_2\ldots x_kx_1$ is a $k$-cycle and we call $F$ a {\it $k$-face}. (A {\it $k$-cycle} is a cycle of length $k$.) Two faces of an outer plane graph are {\it intersecting} if they share at least one common vertex.

If $H$ is a subgraph of $G$ and $f: V(G)\rightarrow [0, k] =\{1, 2, \ldots, k\}$ is a $k$-$L(2, 1)$-labeling of $G$, then we define $f|_{H}: V(H)\rightarrow [0, k]$ to be the restriction of $f$ to $V(H)$; that is, $f|_{H}(u)=f(u)$ for each $u \in V(H)$. Clearly,  $f|_{H}$ is a $k$-$L(2, 1)$-labeling of $H$. Conversely, if $f$ is a 6-$L(2,1)$-labeling of $H$ and $f_1$ a 6-$L(2,1)$-labeling of $G$ such that $f_1|_{H}$ is identical to $f$, then we say that  $f$  can be {\it extended} to $f_1$. In this case, we will simply use $f$ instead of $f_1$ to denote the extended labeling.

\section{Preliminaries}
\label{sec:upper}

A graph $G$ is said to be the {\it 2-sum} of its subgraphs $H_{1}$ and $H_{2}$, written $G = H_{1} \oplus_2 H_{2}$, if $V(G)=V(H_1)\cup V(H_2)$, $E(G) = E(H_{1})\cup E(H_{2})$, $|V(H_{1}) \cap V(H_{2})| = 2$ and $|E(H_{1})\cap E(H_{2})| = 1$.

\begin{lemma}
\label{l1}
Suppose $G$ is a 2-connected outer plane graph with $\Delta=3$, and
$F_1$ and $F_2$ are two 3-faces of $G$. Then $F_1$ and $F_2$ are intersecting if and only
if $G$ is isomorphic to $F_1\oplus_2 F_2$.
\end{lemma}
\begin{proof}
Assume $\partial F_1=u_1u_2u_3$ and $\partial F_2=v_1v_2v_3$. Suppose $F_1$ and $F_2$
are intersecting and let, say, $u_1=v_1$. Since $\Delta=3$, we have $u_2=v_2$ or $u_2=v_3$.
Without loss of generality we may assume $u_2=v_2$. Then $d(u_1)=d(u_2)=3$ and $d(u_3)=d(v_3)=2$. Since $G$ is an outer plane graph, $u_3v_3\notin E(G)$. It follows from the 2-connectivity of $G$ that $G = F_1\oplus_2 F_2$.
Obviously, if $G$ is isomorphic to $F_1\oplus_2 F_2$, then $F_1$ and $F_2$ are intersecting.
\end{proof}

\medskip

\begin{lemma}
\label{le42}
Suppose $G$ is an outerplanar graph with $\Delta=3$ and $v$ is a vertex of $G$ with
$d(v)=2$. Let $P=v_1v_2\ldots v_q$ be a path such that $V(G)\cap V(P)=\emptyset$. Let $G'$ denote the graph obtained from $G$ and $P$ by identifying $v$ with $v_1$. Then $G$ admits a 6-$L(2,1)$-labeling if and only if $G'$ admits a 6-$L(2,1)$-labeling.
\end{lemma}
\begin{proof} If $g$ is a 6-$L(2, 1)$-labeling of $G'$, then $g|_{G}$ is a 6-$L(2,1)$-labeling of $G$.
Conversely, let $f$ be a 6-$L(2,1)$-labeling of $G$. Let $u_1$ and $u_2$ be two neighbors of $v$ in $G$. Denote $a_1=f(u_1)$, $a_2=f(u_2)$ and $b=f(v)$. We extend $f$ to the vertices of
$P$ as follows. The vertex $v_2$ is assigned $f(v_2)\in [0, 6] \setminus \{a_1, a_2, b, b-1, b+1\}$, and for $3\leq j\leq q$, $v_j$ is assigned $f(v_j)\in [0, 6] \setminus \{f(v_{j-2}), f(v_{j-1})-1, f(v_{j-1}), f(v_{j-1})+1\}$.
One can verify that this extension is possible and it defines a 6-$L(2, 1)$-labeling of $G'$.
\end{proof}

\medskip

In Theorems~\ref{le43} and \ref{le44}, we will develop two extension techniques under which a given 6-$L(2, 1)$-labeling of a subgraph of $G$ can be extended to a 6-$L(2, 1)$-labeling of $G$. To this end we define three classes of extendable 6-$L(2, 1)$-labelings as follows. (See Figure \ref{fig:hp} for an illustration.)

\begin{figure}[ht]
\centering
\includegraphics*[height=9cm]{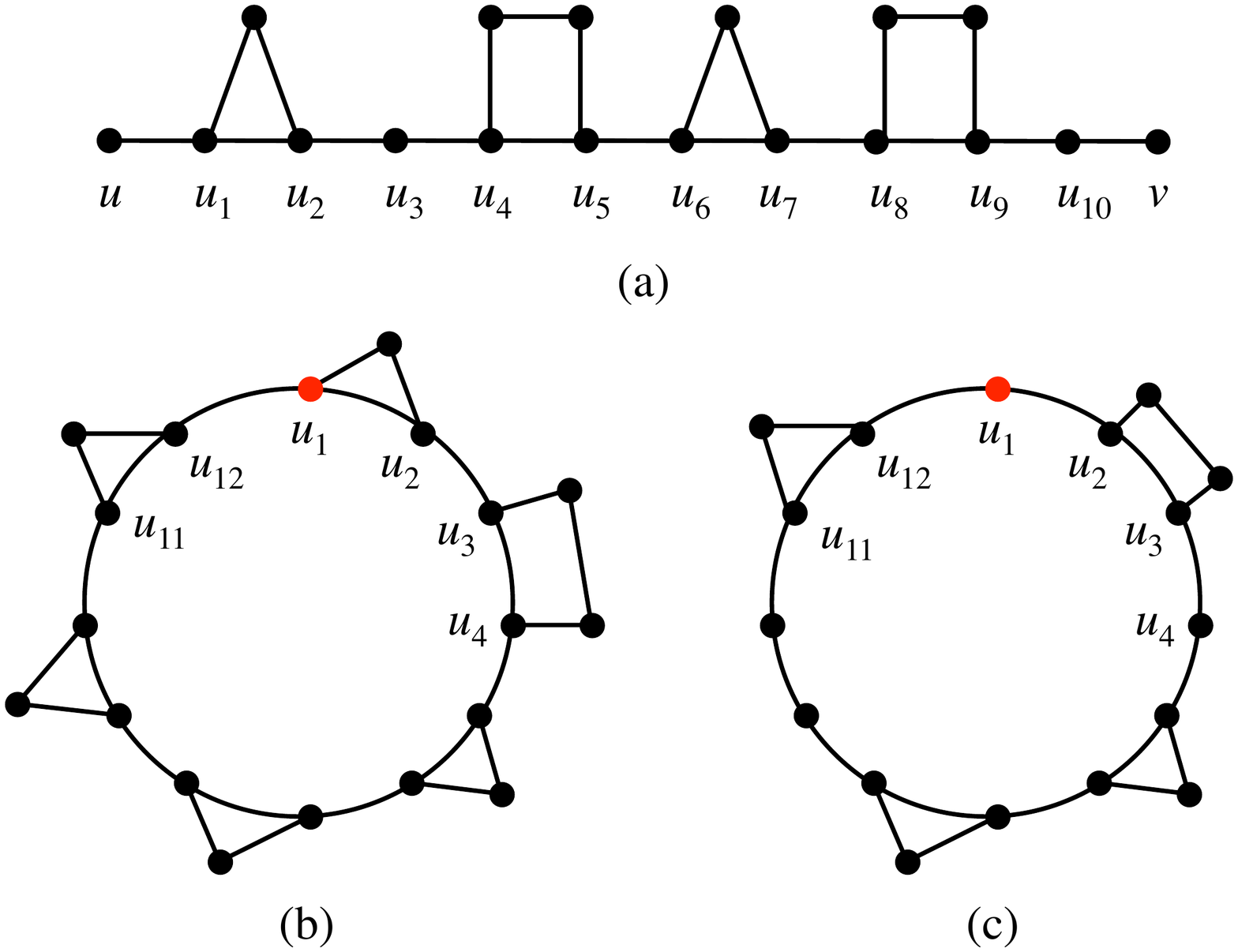}
\vspace{-1.0cm}
\caption{\small An illustration of Definitions \ref{defn:HP} and \ref{defn:KC}: (a) A graph in $\mathscr{H}(P)$ with $l = 10$, $t=4$, $i_1 = 1, i_2 = 4, i_3 = 6$ and $i_4 = 8$; (b) a graph in $\mathscr{K}(C)$ with $l = 12$ and $t=6$, where $C$ has starting vertex $u_1$; (c) a graph in $\mathscr{F}(C)$ with $l = 12$ and $t=5$, where $C$ has starting vertex $u_1$. Note that the graph in (b) is not a member of $\mathscr{F}(C)$.} 
\label{fig:hp}
\end{figure} 

\begin{defn}
\label{defn:HP}
Given a path $P=uu_1u_2\ldots u_lv$ with $l\geq 2$, we define a family of outer plane graphs $\mathscr{H}(P)$ as follows: $G \in \mathscr{H}(P)$ if and only if there exist $i_1, i_2, \ldots, i_{t}$ with $1 \le i_1 <i_2<\ldots < i_{t} < l$ and $i_{j+1} \ge i_{j} + 2$ for $j=1,\ldots,t-1$, such that $G$ can be obtained from $P$ by attaching $t$ paths of length 2 or 3 to $P$ and identifying the two end vertices of the $j$th path to $u_{i_{j}}$ and $u_{i_j+1}$ respectively, for $j=1,\ldots,t$.

A 6-$L(2, 1)$-labeling $f$ of $P$ is called a {\it path-extendable 6-$L(2,1)$-labeling} if $f$ can be extended to a 6-$L(2, 1)$-labeling of every $H\in{\mathscr{H}}(P)$.
\end{defn}

\begin{defn}
\label{defn:KC}
Given a cycle $C=u_1u_2\ldots u_lu_1$ with starting vertex $u_1$ (where $l\geq 3$), we define two families of outer plane graphs, denoted by $\mathscr{K}(C)$ and ${\mathscr F}(C)$, as follows: $G \in \mathscr{K}(C)$ if and only if there exist $i_1, i_2, \ldots, i_{t}$ with $1\leq i_1 <i_2<\ldots < i_{t}\leq l$ and $i_{j+1} \ge i_{j} + 2$ for $j=1,\ldots,t-1$, such that $G$ can be obtained from $C$ by attaching $t$ paths of length 2 or 3 to $C$ and identifying the two end vertices of the $j$th path to $u_{i_{j}}$ and $u_{i_j+1 }$ respectively, for $j=1,\ldots,t$, with the subscripts of $u$'s modulo $l$.

The family ${\mathscr F}(C)$ ($\subseteq \mathscr{K}(C)$) is defined in exactly the same way as $\mathscr{K}(C)$ except that in addition we require $i_1 \geq 2$ and $i_t < l$.  

A 6-$L(2, 1)$-labeling $f$ of $C$ is defined to be a {\it cycle-extendable 6-$L(2, 1)$-labeling of type 1} (respectively, type 2) if $f$ can be extended to a 6-$L(2, 1)$-labeling of every $H\in {\mathscr K}(C)$ (respectively, $H\in {\mathscr F}(C)$).
\end{defn}

We emphasize that $\mathscr{K}(C)$ and ${\mathscr F}(C)$ depend on the starting vertex $u_1$ of $C$, and that in our subsequent discussion $u_1$ should be clear from the context. 

Path-extendable labelings will be used in extension technique 1 while cycle-extendable labelings in extension technique 2.
Not every 6-$L(2,1)$-labeling of $P$ is extendable. For example, if a 6-$L(2, 1)$-labeling $f$ of $P$ satisfies $(f(u_i), f(u_{i+1}), f(u_{i+2}), f(u_{i+3}))=(4, 1, 3, 0)$ for some $1 \le i \le l-3$, then $f$ is not extendable. In fact, if $H$ is obtained from $P$ by identifying the end vertices of a path $Q$ of length 3 to $u_{i+1}$ and $u_{i+2}$ respectively, then $H \in{\mathscr{H}}(P)$ but $f$ cannot be extended to a 6-$L(2, 1)$-labeling of $H$, because we cannot assign labels from $[0, 6]$ to the two middle vertices of $Q$ without violating the $L(2, 1)$-condition. Similarly, if $f$ is a 6-$L(2,1)$-labeling of $P$ such that $(f(x_i), f(x_{i+1}), f(x_{i+2}), f(x_{i+3})) = (0,2, 4,6), (4, 1, 6, 3), (5, 1, 4, 6)$ or $(6, 1, 4, 2)$ for some $1 \le i \le l-3$, then $f$ is not extendable.

For a 6-$L(2,1)$-labeling $f$ of $G$, if $f(u)\in \{1, 3, 5\}$ for a vertex $u$, then there are two possible labels from $\{0, 2, 4, 6\}$ that can be assigned to a neighbor $v$ of $u$ such that $|f(u)-f(v)|\geq 2$. These two labels are called the {\it available neighbor labels}
in $\{0, 2, 4, 6\}$ for $v$ with respective to $f(u)$ (or available neighbor labels
in $\{0, 2, 4, 6\}$ for $f(u)$). Similarly, there are two available neighbor labels in $\{1, 3, 5\}$
for a label $f(u)\in \{0, 6\}$ and one available neighbor label in $\{1, 3, 5\}$ for a label $f(u)\in \{2, 4\}$. The next lemma is straightforward.

\begin{lemma}
\label{le01}
Given a 6-$L(2,1)$-labeling $f$ of an outer plane graph $G$ with $\Delta=3$, if $L_x$ is the set of available neighbor labels in $\{0, 2, 4, 6\}$ for $x\in \{1, 3, 5\}$ and $M_y$ the set of available neighbor labels in $\{1, 3, 5\}$ for $y\in \{0, 2, 4, 6\}$, then the following hold:
\begin{itemize}
\item[\rm (i)] for any $a, b\in \{1, 3, 5\}$ with $a\not=b$, $|L_a\cup L_b|\geq 3$ and $|L_a\setminus L_b|\geq 1$;
\item[\rm (ii)] if $(a, b) = (1, 3)$ or $(3, 5)$, then $|L_a\cap L_b|=1$;
\item[\rm (iii)] for each $a\in \{0, 2, 4, 6\}$, $|M_a|\geq 1$;
\item[\rm (iv)] for any $a, b\in \{0, 2, 4, 6\}$ with $a\not=b$, $|M_a\cup M_b|\geq 2$; moreover, $|M_a\cup M_b|\geq 3$ when $(a,b)\in \{(0, 6), (2, 6)\}$.
\end{itemize}
\end{lemma}

Given a path $u_1u_2\ldots u_{3k}$, we say that we label $u_1, u_2, \ldots, u_{3k}$ using {\em pattern $a\,b\,c$}, or these vertices are labeled using pattern $a\,b\,c$, if, for $0 \le i \le k-1$, $u_{3i+1}$, $u_{3i+2}$ and $u_{3i+3}$ are assigned $a$, $b$ and $c$, respectively.

 \begin{lemma}
 \label{le00}
 \begin{itemize}
\item[\rm (i)] Let $P=uu_1u_2\ldots u_{3k}v$ be a path and $f$ a 6-$L(2,1)$-labeling of $P$. If $f$ labels $u_1, u_2, \ldots$, $u_{3k}$ using pattern $a\,b\,c$, where $\{a, b, c\} \subset \{0, 2, 4, 6\}$, then $f$ is a path-extendable 6-$L(2,1)$-labeling.
\item[\rm (ii)] Let $C=u_1u_2\ldots u_{3k}u_1$ be a cycle and $f$ a 6-$L(2,1)$-labeling of $C$. If $f$ labels $u_1, u_2, \ldots, u_{3k}$ using pattern $a\,b\,c$, where $\{a, b, c\} \subset \{0, 2, 4, 6\}$, then $f$ is a cycle-extendable 6-$L(2,1)$-labeling of type 1.
\item[\rm (iii)] Let $C=u_1u_2\ldots u_{3k}u_{3k+1}u_1$ be a cycle and $f$ a 6-$L(2,1)$-labeling of $C$. If $f$ labels $u_3, u_4, \ldots, u_{3k}$, $u_{3k+1}, u_1$ using pattern $a\,b\,c$, where $\{a, b, c\} \subset \{0, 2, 4, 6\}$, then $f$ is a cycle-extendable 6-$L(2,1)$-labeling of type 2.
\end{itemize}
 \end{lemma}
\begin{proof}
 We prove (a) here. The proofs for (b) and (c) are similar. Let $H \in {\mathscr{H}}(P)$. If a vertex $v$ of $H$ is adjacent to two consecutive vertices of $P$, then $v$ can be assigned the unique label in $\{0, 2, 4, 6\} \setminus \{a, b, c\}$. If two adjacent vertices $u, v$ of $H$ are adjacent to $u_i, u_{i+1}$ respectively, then by Lemma~\ref{le01},  $u$ and $v$ can be assigned labels from $\{1, 3, 5\}$. Thus $f$ can be extended to a 6-$L(2,1)$-labeling of $H$. Since this is true for any $H \in {\mathscr{H}}(P)$, $f$ is extendable.
\end{proof}

\medskip

All 6-$L(2,1)$-labelings of a path $P$ (or a cycle) in this paper fall into two categories for some $a, b, c \in \{0, 2, 4, 6\}$: either all vertices of $P$ are labeled using pattern $a\,b\,c$, or all vertices of $P$ except at most three  vertices at each end of $P$ are labeled using pattern $a\,b\,c$. In the latter case, by Lemma~\ref{le00}, we have the following lemma.

\begin{lemma}
\label{lem01}
 \begin{itemize}
\item[\rm (i)]  Let $P=uv_1\ldots v_tu_1u_2\ldots u_{3k}w_1\ldots w_sv$ be a path and $f$ a 6-$L(2,1)$-labeling of $P$. Suppose $f|_{uv_1\ldots v_tu_1u_2}$ and $f|_{u_{3k-1}u_{3k}w_1\ldots u_sv}$ are  path-extendable 6-$L(2,1)$-labelings of $uv_1\ldots v_tu_1u_2$ and $u_{3k-1}u_{3k}w_1\ldots u_sv$, respectively, where $t\leq 3$ and $s\leq3$. If $f$ labels $u_1, u_2, \ldots, u_{3k}$ using pattern $a\,b\,c$, where $\{a, b, c\} \subset \{0, 2, 4, 6\}$, then $f$ is a path-extendable 6-$L(2,1)$-labeling.

\item[\rm (ii)] Let $C=v_1\ldots v_tu_1u_2\ldots u_{3k}v_1$ be a cycle and $f$ a 6-$L(2,1)$-labeling of $C$. Suppose $f|_{u_{3k-1}u_{3k}v_1\ldots v_tu_1u_2}$ is an extendable 6-$L(2,1)$-labeling of $u_{3k-1}u_{3k}v_1\ldots v_tu_1u_2$, where $t\leq 3$. If $f$ labels $u_1, u_2, \ldots, u_{3k}$ using pattern $a\,b\,c$, where $\{a, b, c\} \subset \{0, 2, 4, 6\}$, then $f$ is a cycle-extendable 6-$L(2,1)$-labeling of type 1.

\item[\rm (iii)] Let $C=v_1\ldots v_tu_3u_4\ldots u_{3k+1}u_1u_2$ be a cycle and $f$ a 6-$L(2,1)$-labeling of $C$. Suppose $f|_{u_{2}v_1\ldots v_tu_3u_4}$ is a path-extendable 6-$L(2,1)$-labeling of $u_{2}v_1\ldots v_tu_3u_4$, where $t\leq 3$. If $f$ labels $u_3, u_4, \ldots, u_{3k+1}, u_1$ using pattern $a\,b\,c$, where $\{a, b, c\} \subset \{0, 2, 4, 6\}$, then $f$ is a cycle-extendable 6-$L(2,1)$-labeling of type 2.
\end{itemize}
\end{lemma}

\section{The lower bound}
\label{sec:lower}

In this section we prove the relatively easy part of Theorem \ref{thm1}, that is, $\lambda(G(l)) \ge 6$ if $l \ge 4$ is not a multiple of $3$. Throughout this section we assume $l\geq 3$ and the vertices $u, v$, $x_1, \ldots, x_l$, $y_1, \ldots, y_l$ of $G(l)$ are as shown in Figure \ref{fig:gl}.

Suppose that $G(l)$ admits a 5-$L(2, 1)$-labeling $f: V(G(l)) \rightarrow [0, 5]$.
Let $H_1$ and $H_2$ be the subgraphs of $G(l)$ induced by $f^{-1}(\{0, 2, 4\})$ and $f^{-1}(\{1, 3, 5\})$ respectively.  The roles of $H_1$ and $H_2$ are symmetric because assigning $5-f(w)$ to $w \in V(G(l))$ yields another 5-$L(2, 1)$-labeling of $G(l)$.

\begin{lemma}
\label{le001}
Every component of $H_1$ or $H_2$ has at least two vertices.
\end{lemma}
\begin{proof} By the symmetry between $H_1$ and $H_2$, it suffices to prove this for any component $H$ of $H_1$.
Suppose to the contrary that $V(H)=\{w\}$ for some $w \in V(G(l))$. Since each vertex of $G(l)$ has degree 2 or 3, we have $|N(w)| = 2$ or $3$. Assume
$N(w)=\{w_1, w_2, w_3\}$ first.  If $f(w)=i$, then $i\in \{0, 2, 4\}$ and $f(w_1), f(w_2),
f(w_3)\in \{1, 3, 5\}$. By the $L(2, 1)$-condition, $f(w_1)$, $f(w_2)$ and $f(w_3)$ are distinct and $\{f(w_1), f(w_2), f(w_3)\}=\{1, 3, 5\}$.
Thus there is at least one $w_j$ such
that $f(w_j)=i+1$ or $i-1$. This is a contradiction as $w_j$ is adjacent to $w$.

Assume then $N(w)=\{w_1, w_2\}$. By the structure of $G(l)$, we may assume $w=u, w_1=x_1$ and $w_2=y_1$. Since $H$ is a component of $H_1$, we have $f(u)\in S_1$ and
$f(x_1), f(y_1)\in S_2$, which implies $f(u)=0$ and $\{f(x_1), f(y_1)\} = \{3, 5\}$. Without loss of generality we may assume $f(x_1)=3$ and $f(y_1)=5$. By the $L(2,1)$-condition we have $f(x_2)=1$.
This implies that $y_2$ cannot be assigned any label from $[0,5]$ without violating the $L(2,1)$-condition, a contradiction again.
\end{proof}

\begin{lemma}
\label{le02}
Let $H$ be a component of $H_1$ or $H_2$. Then the following hold:
\begin{itemize}
\item[\rm (i)] $H$ contains no 4-cycle $x_iy_iy_{i+1}x_{i+1}x_i$, where $1\leq i\leq l$;
\item[\rm (ii)] $H$ cannot contain a 3-vertex and all its neighbors;
\item[\rm (iii)] if $l$ is not a multiple of $3$, then $H$ cannot contain any cycle; if $l$ is a multiple of $3$, then either $H$ itself is a 3-cycle or it does not contain any cycle.
\end{itemize}
\end{lemma}
\begin{proof}
(i) This follows immediately from the $L(2, 1)$-condition.

(ii) Suppose to the contrary that $H$ contains a degree-three vertex $w$ and its
neighbors $w_1, w_2, w_3$. Let $f(w)=j\in S_i$, where $i = 1$ or $2$. Then $f(w_1),
f(w_2), f(w_3)\in S_i\setminus \{j\}$ and hence there exist $s,
t\in \{1, 2, 3\}, s\not=t$ such that $f(w_s)=f(w_t)$. However, this violates the $L(2, 1)$-condition.

(iii) Since the roles of $H_1$ and $H_2$ are symmetric, it suffices to prove the results for $H_1$.
Suppose that a component $H$ of $H_1$ contains a cycle $C$. If $|V(C)|\geq 4$,
then $H$ contains $x_i, x_{i+1}, y_{i+1}, y_i$ for some $i$, contrary to (i).
Thus $|V(C)|=3$ and so by symmetry we may assume
$C=ux_1y_1u$.

Consider $f(u)=0$ first. In this case we may
assume $f(x_1)=2$ and $f(y_1)=4$ by symmetry.
Then $f(x_2)=5$, $f(y_2)=1$, $f(x_{3k})=0$, $f(y_{3k})=3$, $f(x_{3k+1})=2$,
$f(y_{3k+1})=5$, $f(x_{3k+2})=4$ and $f(y_{3k+2})=1$ for $k\geq 1$. Thus $H = C$ is a 3-cycle and moreover $v$ cannot be assigned any label from $[0,5]$ unless $3$ divides $l$.

In the case when $f(u)=2$, we may
assume $f(x_1)=0$ and $f(y_1)=4$ by symmetry. Then $f(x_2)=3$
or 5, and $f(y_2)=1$. When $f(x_2)=3$, we have $f(x_{3})=5$ and $y_3$ cannot be assigned any label from $[0, 5]$. When $f(x_2)=5$, we have $f(x_{3})=3$ or 2, and $y_3$
cannot be assigned any label from $[0, 5]$.

In the case when $f(u)=4$, we may assume $f(x_1)=0$ and $f(y_1)=2$ by symmetry. Then $f(x_2)=3$ and $f(y_2)=5$. This implies that $x_3$ must be assigned 1 and $y_3$ cannot be assigned any label from $[0, 5]$.
\end{proof}

\medskip

A component $H$ of $H_i$ is said to be a {\it path component} if $H$ is a path, where $i=1, 2$. We say that a path component {\it $H$ contains a path $P$} if $V(P)\subseteq V(H)$.

\begin{lemma}
\label{le03}
Let $H$ be a path component of $H_i$, where $i = 1, 2$. Then an end vertex $w$  of $H$ with $d(w)=3$
 must be assigned 0 if $i=1$ and 5 if $i=2$.
\end{lemma}
\begin{proof}
Since the roles of $H_1$ and $H_2$ are symmetric, it suffices to prove the result for $i = 1$.
By Lemmas~\ref{le001}, the length of $H$ is greater than 2.
Assume $H=w_1\ldots w_l$ with $d(w_1)=3$. Let $N(w_1)=\{w_2,
z_1, z_2\}$.  Then $z_1, z_2\notin V(H)$ and $f(z_1), f(z_2)\in \{1, 3, 5\}$. If $f(w_1)=a\not=0$, then
$f(z_j)=a-1$ or $a+1$ for some $j = 1, 2$, which violates the $L(2,1)$-condition.
\end{proof}

\begin{lemma}
\label{le04}
Let $H$ be a path component of $H_1$ or $H_2$. Then the following hold:
\begin{itemize}
\item[\rm (i)] $H$ contains no 3-path $x_ix_{i+1}y_{i+1}y_{i+2}$;
\item[\rm (ii)] if $H$  contains a 2-path $x_iy_{i}y_{i+1}$ (or $y_ix_ix_{i+1}$) such that $x_i$
is an end vertex of a path component of $H_j$ for $j = 1, 2$, then $i=2$;
\item[\rm (iii)] if $H$ contains a 2-path $x_iy_{i}y_{i-1}$ (or $y_ix_ix_{i-1}$) such that $x_i$
is an end vertex of a path component of $H_j$ for $j = 1, 2$, then $i=l-1$.
\end{itemize}
\end{lemma}
\begin{proof}
Since the roles of $H_1$ and $H_2$ are symmetric, we may assume that $H$ is a path component of $H_1$.

(i) Suppose to the contrary that $H$ contains the 3-path $x_ix_{i+1}y_{i+1}y_{i+2}$. By Lemma~\ref{le02}, $y_i, x_{i+2}\in V(H_2)$. By Lemma~\ref{le03}, $f(y_i)=f(x_{i+2})=5$. So $\{f(x_i), f(x_{i+1}), f(y_{i+1}), f(y_{i+2})\}=\{0, 2\}$, which violates the $L(2, 1)$-condition.

(ii) Suppose that $H$ contains such a 2-path $x_iy_{i}y_{i+1}$.  If $i \geq 2$, then $x_{i-1}\in V(H_2)$ by (i). Hence $x_{i+1}\in V(H_2)$ and $y_{i-1}\in V(H_2)$ by Lemma~\ref{le02}. It follows that  $f(x_{i+1})=5$, $f(x_i)=0$, $f(x_{i-1})=3$, $f(y_{i-1})=1$, $f(y_{i+1})=2$ and $f(y_i)=4$.   If $i\geq 3$, then  by symmetry $x_{i-2}$ must be assigned 5, but then $y_{i-2}$ cannot be assigned any label in $[0, 5]$. If $i=1$, then $f(x_1)=0$, $f(y_1)=2$, $f(y_2)=4$ and $f(u)=5$ by symmetry, but then $x_2$ cannot be assigned any label $[0, 5]$. Therefore, $i=2$.

(iii) The proof is similar to that of (ii).
\end{proof}

\begin{lemma}
\label{le05}
Let $H$ be a  component of $H_1$ or $H_2$. Then $H$ is one of the following:
\begin{itemize}
\item[\rm (i)] a 3-cycle;
\item[\rm (ii)] the path $x_3x_4\ldots x_{l}$ {\rm(}or $y_3y_4\ldots y_{l}${\rm )};
\item[\rm (iii)] the path $x_3x_4\ldots x_{l-1}y_{l-1}$ {\rm(}or $y_3y_4\ldots y_{l-1}x_{l-1}${\rm )};
\item[\rm (iv)] the path $x_2y_2\ldots y_{l-2}$ {\rm (}or $y_2x_2\ldots x_{l-2}${\rm )};
\item[\rm (v)] the path $ux_1x_2\ldots x_{l}$ {\rm (}or $uy_1y_2\ldots y_{l}${\rm )};
\item[\rm (vi)] the path $y_1y_2\ldots y_{l}v$ {\rm (}or $x_1x_2\ldots x_{l}v${\rm )};
\item[\rm (vii)]  the path $x_2y_2\ldots y_{l}v$ {\rm (}or $y_2x_2\ldots x_{l}v${\rm )}.
\end{itemize}
\end{lemma}
\begin{proof} By symmetry, we may assume that $H$ is a component of $H_1$ and $x_i\in V(H)$ with minimum subscript $i$. By Lemma~\ref{le001}, $x_{i+1}\in V(H)$ or $y_i\in V(H)$.

Assume first that $x_{i+1}\in V(H)$ and $y_i\notin V(H)$. Let $x_j\in V(H)$ be such that $j$ is maximum. By Lemma~\ref{le04}, $x_i, x_{i+1}, \ldots, x_j\in V(H)$. If $i=1$, then $y_1\in V(H_2)$ by Lemma~\ref{le04}. By symmetry, $j=l$. By Lemma~\ref{le02}, the vertices of $H$ are assigned $0\,2\,4\,0\,2\,4\ldots 0\,2\,4\,0$ or $0\,4\,2\,0\,4\,2\ldots 0\,4\,2\,0$ successively. We thus conclude that $H$ is the path $ux_1x_2y_2\ldots x_{l}$ or $x_1x_2\ldots x_{l}v$ and hence (v) or (vi) holds.
If $i\geq 2$ and $i\not=3$, then $x_{i-1}\in V(H_2)$. Since $y_i\in V(H_2)$, $y_{i-1}\notin V(H_2)$ by Lemma~\ref{le04}. Thus $y_{i-1}\in V(H_1)$. It follows that $x_{i-1}$ is an end vertex of a path of $H_2$ and so is $y_i$. By Lemma~\ref{le03}, $x_{i-1}$ and $y_i$ are assigned 5, contradicting the $L(2, 1)$-condition. Thus, assume that $i=3$. By Lemma~\ref{le04}, $j=l-1$ or $l$. If $j=l-1$, by Lemma~\ref{le04}, $y_{l-1}\in V(H)$, $y_l, y_{l-2}\in V(H_2)$. We thus conclude that (iii) holds.  If $j=l$, then  (ii) holds.

Next we assume that $x_{i+1}\notin V(H)$ and $y_i\in V(H)$. If $i=1$,  then $x_{2}, y_2\in V(H_2)$ by Lemma~\ref{le04}. It follows that $H$ is a 3-cycle and (i) holds. So we assume $i\geq 2$. If $i=2$, then $x_1\in V(H_2)$ and  $y_1\in V(H_2)$ by Lemma~\ref{le04}.  Thus $x_3\in V(H)$ or $y_3\in V(H)$. By symmetry, we may assume $x_3\in V(H)$. By Lemma~\ref{le04}, we may assume $H=y_2x_2\ldots x_j$ such that $j$ is maximum. By Lemma~\ref{le04}, $i=2$ and $j\geq l-2$. Let $j=l-2$. In this case, (iv) holds. If $j=l-1$, then $y_{l-1}\in V(H_2)$ by Lemma~\ref{le04}. Thus, $y_{l},x_{l}\in V(H_2)$, contradicting Lemma~\ref{le04}. If $j=l$, by Lemma~\ref{le02}, $y_3, y_4, \ldots, y_l\in V(H_2)$, which form a path. Note that the vertices of $H$ are assigned $0\,2\,4\,0\,2\,4\ldots 0\,2\,4\,0$ or $0\,4\,2\,0\,4\,2\ldots 0\,4\,2\,0$ sequentially, and $y_3, y_4, \ldots, y_l$ should be assigned $5\,3
\,1\,5\,3\,1\ldots 5\,3\,1\,5$ or $5\,1\,3\,5\,1\,3\ldots 5\,1\,3\,5$ sequentially, which implies $v\in V(H)$. We conclude that (ii) holds.

Finally, we assume that $x_{i+1}\in V(H)$ and $y_i\in V(H)$. Then $i=2$ by Lemma~\ref{le04}. Let $x_{3},\ldots, x_j\in V(H)$ such that $j$ is the maximum subscript. Then $x_{j+1}, y_{j+1}, y_j, \ldots, y_3\in V(H_2)$. By Lemma~\ref{le04}, $j=l-2, l-1, l$. If $j=l-1$, then $y_3, \ldots y_l, x_l, v\in V(H_2)$, which implies that $v, y_l, x_l, x_{l-1}$ cannot be assigned labels from $\{1, 3, 5\}$ without violating the $L(2, 1)$-condition, a contradiction. Thus, $j=l-2$ or $l$, which means that (iv) or (vii) holds.
\end{proof}

\begin{thm}
\label{lower}
Let $l \ge 4$ be an integer which is not a multiple of $3$. Then $\lambda (G(l)) \geq 6$.
\end{thm}
\begin{proof}
Suppose to the contrary that $G(l)$ admits a 5-$L(2, 1)$-labeling $f$. Recall that the vertices $u, v, x_1, \ldots x_l, y_1, \ldots y_l$ of $G(l)$ are as shown in Figure \ref{fig:gl}. By symmetry, we may assume that $H$ is a component of $H_1$ containing $u$. Then $H$ is as in (i) or (v) of Lemma~\ref{le05}. If $H$ is as in (i), then by Lemmas~\ref{le04} and \ref{le05}, the path in (ii) or (iii) is a component $K$ of $H_1$. If $K$ is as in (ii), then the path in (vii) is a component of $H_2$. By Lemma~\ref{le03},  the vertices of $K$ are assigned $0\,2\,4\,0\,2\,4\ldots 0\,2\,4\,0$ or $0\,4\,2\,0\,4\,2\ldots 0\,4\,2\,0$ sequentially.  By Lemma~\ref{le05}, $l-2=3k+1$ and $l$ is a multiple of $3$, a contradiction. The proof is similar in the case when $K$ is as in (iii). If $H$ is as in (v), then $x_l$ and $u$ must be assigned 0 by Lemma~\ref{le03}. Thus the vertices of $x_l, x_{l-1}, \ldots, x_1, u$ must be assigned $0\,2\,4\ldots, 0\,2\,4,\ldots,0\,2\,4\,0$ or $0\,4\,2\ldots 0\,4\,2\ldots,0\,4\,2\,0$ sequentially.  By Lemma~\ref{le05}, $l+1=3k+1$ and $l$ is a multiple of $3$, a contradiction again.
\end{proof}

\section{Extension technique 1}
\label{sec:ext1}

\begin{notn}
\emph{Let $C=u_1u_2\ldots u_lu_1$ be a cycle of length $l \ge 4$, and let $v_1$ and $v_2$ be two additional vertices not on $C$. Define $H$ be the graph obtained from $C$ by adding the edges $u_1v_1$ and $u_2v_2$. Denote $P=v_1u_1u_2v_2$, which is a path of $H$. Let $H_1$ denote the subgraph of $H$ induced by $\{v_1, v_2, u_1, u_2, u_3, u_l\}$.}

\emph{Throughout this section, $H, P$ and $H_1$ are as above, and $f$ is a fixed path-extendable 6-$L(2, 1)$-labeling of $P$.}
\end{notn}

The main result in this section, Theorem~\ref{le43}, states that any given path-extendable 6-$L(2, 1)$-labeling of $P$ can be extended to a 6-$L(2, 1)$-labeling of $H$. To establish this result we need to prove a few lemmas first.

\begin{lemma}
\label{le43-1}
Suppose $\{f(u_1), f(u_2)\}\cap \{1, 3, 5\}\not=\emptyset$ and $f$ can be extended to a 6-L(2,1)-labeling  of $H_1$ such that $f(u_3), f(u_l)\in \{0, 2, 4, 6\}$. Suppose further that $f(u_3)=f(u_l)$ if and only if $l\equiv 0$ (mod 3). Then $f$ can be extended to a 6-$L(2, 1)$-labeling $f_1$ of $H$ such that $f_1|_{C-\{u_1u_2\}}$ is a path-extendable 6-$L(2,1)$-labeling.
\end{lemma}
\begin{proof}
Denote $f(u_3)=a$ and $f(u_l)=b$. Our assumption means $|\{f(u_1), f(u_2)\}\cap \{1, 3, 5\}|=1$ or $2$. Let us first consider the latter case, that is, $\{f(u_1), f(u_2)\}\subseteq \{1, 3, 5\}$. Since $\{0, 2, 4, 6\} \setminus \{a, b\}\not=\emptyset$, we can choose $c\in \{0, 2, 4, 6\} \setminus \{a, b\}$. If $l\equiv 2$ (mod 3), then we label $u_4, u_5, \ldots, u_l$ using pattern $c\,b\,a$, and label $u_{l-1}$ by $c$. In the case $l\equiv 1$ (mod 3), if $l=4$, there is nothing to prove; if $l\geq 5$, then we label $u_4, \ldots, u_{l-1}$ using pattern $b\,c\,a$.
In the case $l\equiv 0$ (mod 3), we have $f(u_3)=f(u_l)=a$ by our assumption. Choose $b, c\in \{0, 2, 4,6\} \setminus \{a\}$ such that $b\not=c$. We label $u_4, u_5$ by $b, c$ respectively, and $u_6, u_7, \ldots, u_{l-1}$ using pattern $abc$.

Assume $|\{f(u_1), f(u_2)\}\cap \{1, 3, 5\}|=1$ from now on.  By symmetry
we may assume $f(u_1)\in \{1, 3, 5\}$ and $f(u_2)\in \{0, 2, 4, 6\}$.

Consider the case $l\equiv 0$ (mod 3) first. In this case, $f(u_3)=f(u_l)=a$. Denote $b=f(u_2)$ and take $c\in \{0, 2, 4, 6\} \setminus \{a, b\}$. We label $u_4, u_5$ by $c, b$ respectively and
$u_6, \ldots, u_{l-1}$ using pattern $a\,c\,b$.

Now we consider the case $l\equiv 2$ (mod 3). Since $f(u_1)$ is  a common available neighbor label in $\{1, 3, 5\}$ for $f(u_2)$ and $f(u_l)$, $(f(u_2), f(u_l))\in \{(6, 0), (0, 6), (0, 2), (2, 0), (4, 6), (6, 4)\}$. If $(f(u_2), f(u_l))\in \{(0, 2), (2, 0), (4, 6), (6, 4)\}$, we label $u_4$ by $c\in \{0, 2, 4, 6\} \setminus \{a, b, f(u_2)\}$, and $u_5, \ldots, u_{l-1}$ using pattern $b\,a\,c$.  We are left with the case where $(f(u_2), f(u_l))=(0,6)$ or $(6, 0)$, for which $f(u_1)=3$ and $f(v_1)\in \{1, 5\}$.

Suppose $(f(u_2), f(u_l))=(0,6)$. If $f(v_2)\not=5$, then we re-assign 5 to $u_3$, assign 1 to $u_4$, and label
$u_5, u_6, \ldots, u_{l-1}$ using pattern $6\,0\,2$. Assume $f(v_2)=5$. If $f(v_1)=1$, then  we re-assign 5 to $u_l$, 2 to $u_{l-1}$, 4 to $u_{l-2}$ and label
$u_{l-3}, u_{l-4}, \ldots, u_{3}$ using pattern $0\,2\,4$; if $f(v_1)=5$, then  we re-assign 1 to $u_l$, 6 to $u_{l-1}$, 2 to $u_{l-2}$ and label
$u_{l-3}, u_{l-4}, \ldots, u_{3}$ using pattern $0\,6\,2$.

Suppose $(f(u_2), f(u_l))=(6, 0)$. If $f(v_2)\not=1$, then we
re-assign 1 to $u_3$, assign 5 to $u_{4}$, and label
$u_{5}, u_{6}, \ldots, u_{l-1}$ using pattern $0\,4\,2$. Assume $f(v_2)=1$. In this case, since $f(u_1)=3$, $f(v_1)\in \{1, 5\}$. If $f(v_1)=5$, then we re-assign 2 to $u_3$, assign 5 to $u_{4}$ and 1 to $u_{5}$, and label $u_{6}, u_{7}, \ldots, u_l$ using pattern $6\,4\,0$ when $l\geq 8$; we re-assign 2 to $u_3$ and 0 to $u_{5}$, and assign 4 to $u_{4}$ when $l=5$. If $f(v_2)=f(v_1)=1$, then we re-assign 5 to $u_l$, assign 0 to $u_{l-1}$ and 4 to $u_{l-2}$, and label $u_{l-3}, u_{l-4}, \ldots, u_3$ using pattern $6\,0\,4$.

Finally, in the case $l\equiv 1$ (mod 3), if $l=4$, there is nothing to prove; if $l\geq 5$, then we label $u_4, \ldots, u_{l-1}$ using pattern $b\, f(u_2)\, a$.

In each possibility above, by Lemmas~\ref{le00} and \ref{lem01}, we obtain a 6-$L(2, 1)$-labeling $f_1$ of $H$ with the desired property.
\end{proof}

\begin{lemma}
\label{le43-2}
If $l\equiv 0$ (mod 3) and $\{f(u_1), f(u_2)\}\cap \{1, 3, 5\}\not=\emptyset$, then $f$ can be extended to a 6-$L(2, 1)$-labeling $f_1$ of $H$ such that $f_1|_{C-\{u_1u_2\}}$ is a path-extendable 6-$L(2,1)$-labeling.
\end{lemma}
\begin{proof} We first prove:

\medskip
\n{\bf Claim.} If we can assign $u_3$  a label from $\{1, 3, 5\}$ and assign $u_l$ a label from $\{0, 2, 4, 6\}$ such that they have no conflict with the existing labels, then $f$ can be extended to a 6-$L(2, 1)$-labeling $f_1$ of $H$ such that $f_1|_{C-\{u_1u_2\}}$ is a path-extendable 6-$L(2,1)$-labeling.

\medskip
\n{\it Proof of the Claim.}  Assume first that $f(u_1), f(u_2)\in \{1, 3, 5\}$. By the $L(2, 1)$-condition, $f(u_3)\not=f(u_1)$. By Lemma~\ref{le01} (a), $f(u_3)$ has an available neighbor label $a$ in $\{0, 2, 4, 6\}$  which is not an available neighbor label for $f(u_1)$. It follows that $a\not= f(u_l)$.  Let $c \ne a$ be an available neighbor label for $f(u_3)$, and let $b\in\{0, 2, 4, 6\} \setminus \{f(u_l), a, c\}$. We label $u_4, u_5, \ldots, u_l$ using pattern $a\, b\, f(u_l)$.

Now we assume $|\{f(u_1), f(u_2)\}\cap \{1, 3, 5\}|=1$. Then $f(u_1)\in \{1, 3, 5\}$ or $f(u_2)\in \{1, 3, 5\}$.

\medskip
\n{\bf Case 1}:~$f(u_1)\in \{1, 3, 5\}$.
\medskip

Then $f(u_2)\in \{0,2,4,6\}$. By our assumption,  $f(u_3)\in \{1,3, 5\}$. Since $f(u_1), f(u_3)\in \{1, 3, 5\}$, $f(u_2)$ is an available neighbor label for both $f(u_1)$ and $f(u_3)$. This leads to $f(u_2)=6$ or 0. Moreover, when $f(u_2)=6$, $(f(u_1), f(u_3))\in \{(3,1), (1, 3)\}$; when $f(u_2)=0$, $(f(u_1), f(u_3))\in \{(3,5), (5, 3)\}$. If $(f(u_1), f(u_3))=(5,3)$, then by assumption, $f(u_l)\in \{0, 2, 4, 6\}$, which implies $f(u_l)=2$ and $f(v_1)\not=2$. In this case,  we label $u_4,\ldots, u_l$ using pattern $6\,0\,2$. If $(f(u_1), f(u_3))=(1,3)$, then by assumption, $f(u_l)\in \{0, 2, 4, 6\}$, which implies $f(u_l)=4$ and $f(v_1)\not=4$. Thus, we label $u_4,\ldots, u_l$ using pattern $0\,6\,4$.

Consider $(f(u_1), f(u_3))=(3,5)$. By our assumption, $f(u_l)\in \{0, 2, 4, 6\}$, which implies $f(u_l)=6$ and $f(v_1)\not=6$. If $f(v_2)\not=6$, then we re-assign 6 to $u_3$, and label $u_4, \ldots, u_l$ using pattern $2\,0\,6$. If $f(v_2)=6$, then we re-assign 4 to $u_3$, and label $u_4, \ldots, u_l$ using pattern $2\,0\,6$.

Consider $(f(u_1), f(u_3))=(3,1)$. If $f(v_2)\not=0$, then we re-assign 0 to $u_3$, and label $u_4, \ldots, u_l$ using pattern $4\,6\,0$. If $f(v_2)=0$, then we re-assign 2 to $u_3$, and label $u_4, \ldots, u_l$ using pattern $4\,6\,0$.

\medskip
\n{\bf Case 2}:~$f(u_2)\in \{1, 3, 5\}$.
\medskip

Then $f(u_1)\in \{0, 2, 4, 6\}$. If $f(u_1)$ is an available neighbor label in $\{0, 2, 4, 6\}$ for $f(u_3)$, let $b$ be the other available neighbor label in $\{0, 2, 4, 6\}$ for $f(u_3)$, that is, $b\not=f(u_1)$. We label $u_4, \ldots, u_{l}$ using pattern $f(u_1)\, a\, c$, where $a\in \{0, 2, 4,6\} \setminus \{f(u_1), b\}$ and $c\in \{0, 2, 4, 6\} \setminus \{f(u_1),a\}$. In what follows we assume that
\begin{equation}
\label{eq1}
\mbox{ $f(u_1)$ is not an available neighbor label for $f(u_3)$.}
\end{equation}
Assume $f(u_2)=1$ first. Then  $f(u_1)\in \{4, 6\}$. Assume first that $f(u_1)=6$. If $f(v_2)\not=3$, then we can re-assign 3 to $u_3$. Thus $f(u_1)$ is an available neighbor label in $\{0, 2, 4, 6\}$ for $f(u_3)$, which contradicts (\ref{eq1}). If $f(v_2)=3$, then $u_3$ is assigned 5 by assumption. In this case, $f(v_1)\not=4$ since $f$ is a path extendable labeling of $P$. Thus, we re-assign 4 to $u_3$, and label $u_4, \ldots, u_l$ using patten $6\,2\,4$.

Therefore, we may assume that $f(u_1)=4$.
Since $f(u_3)\in \{1, 3, 5\}$ and $f(u_2)=1$, we have $f(u_3)\in \{3, 5\}$. Consider $f(u_3)=3$. If $f(v_1)=0$, then label $u_4$ by 5, $u_5, \ldots, u_{l-2}$ using pattern $2\,6\,4$, and label $u_{l-1}$ and $u_{l}$ by 2, 6 respectively. If $f(v_1)\not=0$, then label $u_4, \ldots, u_{l}$ using pattern $6\,2\,0$. Consider $f(u_3)=5$. In this case, label $u_4$ by 3 and $u_5, \ldots, u_{l-2}$ using pattern $6\, a\, 4$, where $a\in \{0, 2, 4, 6\} \setminus \{6, 4, f(v_1)\}$, and label $u_{l-1}$ and $u_{l}$ by 6 and $a$ respectively.

Next we assume $f(u_2)=3$. Then $f(u_1)\in\{0, 6\}$ and we can re-assign $u_3$ a label such that $f(u_1)$ is an available neighbor label in $\{0, 2, 4, 6\}$ for both $f(u_2)$ and $f(u_3)$, which contradicts (\ref{eq1}).

Finally, we assume $f(u_2)=5$. Then $f(u_1)\in \{0, 2\}$. We first assume that $f(u_1)=0$. If $f(v_2)\not=3$, then we can re-label $u_3$ by 3. Thus $f(u_1)$ is an available neighbor label in $\{0, 2, 4, 6\}$ for $f(u_3)$, which contradicts (\ref{eq1}).  Assume $f(v_2)=3$. If $f(v_1)\not=2$, then we re-label $u_3$ by 2 and label $u_4, \ldots u_l$ using patten $0\,4\, 2$; if $f(v_1)=2$, then we re-label $u_3$ by 2 and label $u_4, \ldots u_l$ using patten $0\,6\, 4$. Thus, we assume that $f(u_1)=2$.
In the case when $f(u_3)=1$, label $u_4$ by 3 and $u_5, \ldots, u_{l-2}$ using pattern $0\, a\, 2$, where $a\in \{0, 2, 4, 6\} \setminus \{0, 2, f(v_1)\}$, and label $u_{l-1}$ and $u_{l}$ by 0 and $a$ respectively. In the case when $f(u_3)=3$, if $f(v_1)\not=0$, then label $u_4$ by 1 and $u_5, \ldots, u_{l-3}$ using pattern $4\,0\,2$, and label $u_{l-2}$ and $u_{l-1}$ by 4 and $0$ respectively; if $f(v_1)=0$, then label $u_4, \ldots, u_{l}$ using pattern $0\,4\,6$.

By Lemmas~\ref{le00} and \ref{lem01}, in each possibility above, we obtain a 6-$L(2, 1)$-labeling $f_1$ of $H$ with the desired property. This completes the proof of the claim.
$\square$

\medskip

We are now ready to prove our lemma. Assume first that $\{f(u_1), f(u_2)\} \subset \{1, 3, 5\}$. By symmetry we may assume $(f(u_1), f(u_2))\in \{(1, 3), (1, 5), (3, 5)\}$. If $\{f(v_1), f(v_2)\}\cap \{0, 2, 4, 6\}\not=\emptyset$, then one of $u_3$ and $u_l$ can be assigned a label from $\{0, 2, 4, 6\}$ and the other a label from $\{1, 3, 5\}$. By Claim 1, $f$
can be extended to a 6-$L(2, 1)$-labeling $f_1$ of $H$ such that
 $f_1|_{C-\{u_1, u_2\}}$ is an extendable 6-$L(2,1)$-labeling.
So we may assume $f(v_1), f(v_2)\in \{1, 3, 5\}$. Since $f$ is an extendable 6-$L(2, 1)$-labeling of $P$, $(f(u_1), f(u_2))\in \{(3, 5), (1, 3)\}$. Then $f(u_1)$ and $f(u_2)$ have a common available neighbor label $c$ in $\{0, 2, 4, 6\}$. Thus both $u_3$ and $u_l$ can be assigned $c$. By Lemma~\ref{le43-1}, $f$ can be extended to a 6-$L(2, 1)$-labeling $f_1$ of $H$ such that $f_1|_{C-\{u_1u_2\}}$ is an extendable 6-$L(2,1)$-labeling.

Next we assume $|\{f(u_1), f(u_2)\}\cap \{1, 3, 5\}|=1$.  By symmetry we may assume $f(u_1)\in \{1, 3, 5\}$ and $f(u_2)\in \{0, 2, 4, 6\}$. Then $u_l$ can be assigned a label from $\{1, 3, 5\}$ and $u_3$ a label from $\{0, 2, 4, 6\} \setminus \{f(u_2), f(v_2)\}$. By Claim 1, $f$
can be extended to a 6-$L(2, 1)$-labeling $f_1$ of $H$ such that
 $f_1|_{C-\{u_1u_2\}}$ is a path-extendable 6-$L(2,1)$-labeling.
\end{proof}

\medskip
The next lemma can be easily verified. It will be used in the proof of Lemma~\ref{le43-3}.

\begin{lemma}\label{le43-3-1} Let $W=w_1w_2w_3w_4$ be a path. If $(a, b)\in \{(6,0), (0, 6), (2, 6), (6, 2), (0, 4), (4, 0), (4, 2), (2, 4)\}$, then there is a path-extendable 6-$L(2,1)$-labeling $f$ of $W$ such that
\begin{itemize}
\item[\rm (i)] $f(w_1)=b$ and $f(w_4)=a$;
\item[\rm (ii)] $f(w_i)\in \{1, 3, 5\}$ for $i=2, 3$;
\item[\rm (iii)] for every $H\in {\mathscr {H}}(W)$, each vertex of $V(H) \setminus V(W)$  can be assigned a label from $[0,6] \setminus \{a, b, f(w_2), f(w_3)\}$  such that the resulting labeling is a path-extendable 6-$L(2, 1)$-labeling.
\end{itemize}
\end{lemma}

\begin{lemma} \label{le43-3}
If $|\{f(v_1), f(v_2),  f(u_1), f(u_2)\}\cap \{1, 3, 5\}|\leq 1$, then $f$ can be extended to a 6-$L(2, 1)$-labeling $f_1$ of $H$ such that $f_1|_{C-\{u_1u_2\}}$ is a path-extendable 6-$L(2,1)$-labeling.
\end{lemma}
\begin{proof}
We distinguish the following two cases.

\medskip

\n{\bf Case 1.} $|\{f(u_1), f(u_2)\}\cap \{1, 3, 5\}|=1$.

\medskip

We may assume $l\equiv 1$ or 2 (mod 3) for otherwise the result is true by Lemma~\ref{le43-2}. By symmetry we may assume $f(u_1)\in \{1, 3, 5\}$. Then $f(v_1), f(v_2), f(u_2)\in \{0, 2, 4, 6\}$.
Consider the case $l\equiv 2$ (mod 3) first.  Since $f(v_1), f(u_2)\in \{0, 2, 4, 6\}$, $f(v_1)$ and $f(u_2)$ are two available neighbor labels for $f(u_1)$. By Lemma~\ref{le01}, there is a label in $a\in \{1, 3, 5\} \setminus \{f(u_1)\}$ such that $f(v_2)$ is an available neighbor label for $a$. Assign $a$ to $u_l$  and $f(v_2)$ to $u_{l-1}$. Choose $b\in \{0, 2, 4, 6\} \setminus \{f(v_2), f(u_2)\}$. Assign $b$ to $u_3$ and label $u_{4}, \ldots, u_{l-2}$ using pattern $f(v_{2})\,f(u_{2})\,b$.

Next assume $l\equiv 1$ (mod 3). The vertex $u_l$ can be assigned a label
$f(u_l)\in \{1, 3, 5\} \setminus \{f(u_1)\}$ such that $f(u_l)$ has an available neighbor label
$a\in \{0, 2, 4, 6\} \setminus \{f(u_2), f(v_2)\}$. Assign $a$ to $u_{l-1}$ and $u_{l-4}$, $f(u_2)$ to $u_{l-2}$, $f(v_2)$ to $u_{l-3}$, and label $u_{l-5}, \ldots, u_3$ using pattern $f(u_{2})\, f(v_{2})\, a$.

\medskip
\n{\bf Case 2.} $\{f(u_1), f(u_2)\}\cap \{1, 3, 5\}=\emptyset$.
\medskip

Since $|\{f(v_1), f(v_2), f(u_1), f(u_2)\}\cap \{1, 3, 5\}|\leq 1$, $|\{f(v_1), f(v_2)\}\cap \{1, 3, 5\}|\leq 1$. We will consider three subcases: $l\equiv 0, 1, 2$ (mod 3). In the case when $l\equiv 0$, we will consider two subcases: $|\{f(v_1), f(v_2)\}\cap \{1, 3, 5\}|=0$ or 1. In the case when $l\not\equiv 0$, we do not consider any subcase explicitly. In the case when $l\equiv 0$ (mod 3) and $|\{f(v_1),  f(v_2)\}\cap \{1, 3, 5\}|=1$,  we may assume without loss of generality that $f(v_1)\in \{1, 3, 5\}$ and $f(u_1), f(u_2), f(v_2)\in \{0, 2, 4, 6\}$.

Consider the case $l\equiv 0$ (mod 3). If $|\{f(v_1), f(v_2)\}\cap \{1, 3, 5\}|=1$, choose $a\in \{0, 2, 4, 6\} \setminus \{f(u_1), f(u_2), f(v_2)\}$. Then assign $a$ to $u_3$ and label $u_4, u_5, \ldots, u_l$ using pattern $f(u_1)\, f(u_2)\, a$.

If $|\{f(v_1), f(v_2)\}\cap \{1, 3, 5\}|=0$,  by symmetry we may assume $f(u_1)<f(u_2)$. If $(f(u_1), f(u_2))\in \{(0, 6), (0, 2)\}$, then  $f(u_1)$ and $f(u_2)$ have a common available neighbor label $a\in \{1, 3, 5\}$ such that $u_3$ can be assigned $a$, and $u_4, u_5, \ldots, u_l$ can be labeled using pattern $f(u_1)\ f(u_2)\ f(v_2)$. We now assume that $(f(u_1), f(u_2))\in \{(4, 6), (2, 6)\}$. If $f(v_1)\not=f(v_2)$, then we assign $u_3$ by 3, $u_4$ by 1, $u_5$ by 6, label $u_6, \ldots, u_{l-1}$ using patten $f(v_2)\, f(u_1)\, 6$ and assign $u_l$ by $f(v_2)$. If $f(v_1)=f(v_2)$, let $a\in \{0, 2, 4, 6\}\setminus \{f(u_1), f(u_2), f(v_1)\}$. We assign $u_3$ by 3, $u_4$ by 1, $u_5$ by 6, label $u_6, \ldots, u_{l-1}$ using patten $a\, f(u_1)\, 6$ and assign $u_l$ by $a$.
 It remains to consider the case $(f(u_1), f(u_2))\in \{(0, 4),  (2, 4)\}$. In this case, since $f$ is an extendable $L(2, 1)$-labeling of $P$, we have $f(v_1)=f(v_2)$. In the case $(f(u_1), f(u_2))=(0, 4)$, if $f(v_1)=f(v_2)=6$,  then  label $u_3, u_4$ by 1, 3 respectively, label $u_5,u_6, \ldots, u_{l-3}$ using pattern $6\,2\,0$, and label $u_{l-2}, u_{l-1}$ by 6 and 2, respectively; if $f(v_1)=f(v_2)=2$,  then  label $u_3, u_4$ by 1, 5 respectively, label $u_5,u_6, \ldots, u_{l-3}$ using pattern $2\,6\,0$, and label $u_{l-2}, u_{l-1}$ by 2 and 6, respectively.
In the case $(f(u_1), f(u_2))=(2, 4)$, if $f(v_1)=f(v_2)=6$, then label $u_3, u_4$ by 0, 3 respectively,  label $u_5,u_6, \ldots, u_{l-3}$ using pattern $6\,0\,2$, and label $u_{l-2}, u_{l-1}$ by 6 and 0, respectively; if $f(v_1)=f(v_2)=0$, then label $u_3, u_4$ by 1, 5 respectively,  label $u_5,u_6, \ldots, u_{l-3}$ using pattern $0\,6\,2$, and label $u_{l-2}, u_{l-1}$ by 0 and 6, respectively.

Consider the case $l\equiv 1$ (mod 3). If $(f(u_1), f(u_2)) \in \{(6,0), (0, 6), (2, 6), (6, 2), (0, 4), (4, 0), (4, 2), (2, 4)\}$, then by Lemma~\ref{le43-3-1} we have an extendable 6-$L(2, 1)$-labeling of path $u_2u_3u_4u_5$ such that $f(u_5)=f(u_1)$. If $f(v_1)\not=f(v_2)$, then label $u_6, \ldots, u_{l-2}$ using pattern $f(u_2)\, f(v_2)\, f(u_1)$, and label $u_{l-1}$ and $u_{l}$ by $f(u_2)$ and $f(v_2)$, respectively. If $f(v_1)=f(v_2)$, let $a\in \{0, 2, 4, 6\} \setminus \{f(v_2), f(u_1), f(u_2)\}$. In this case, label $u_6, \ldots, u_{l-2}$ using pattern $f(u_2)\, a\, f(u_1)$, and $u_{l-1}$ and $u_{l}$ by $f(u_2)$ and $a$, respectively. Thus we assume $(f(u_1), f(u_2))\in \{(0, 2), (2, 0), (6, 4), (4, 6)\}$. By Lemma~\ref{le43-3-1}, we have an extendable 6-$L(2, 1)$-labeling of path $u_2u_3u_4u_5$ such that $f(u_5)=f(u_1)$.    If $f(v_1)=f(v_2)$, let $a\in \{0, 2, 4, 6\} \setminus \{f(u_1), f(u_2), f(v_2)\}$. In this case,  label $u_l, \ldots, u_{8}$ using pattern $a\, f(u_2)\, f(u_1)$, and label $u_7, u_6$ by $a, f(u_2)$, respectively. If $f(v_1)\not=f(v_2)$, then label $u_l, \ldots, u_{8}$ using pattern $f(v_2)\, f(u_2)\, f(u_1)$, and label $u_{7}, u_6$ by $f(v_2)$ and $f(u_2)$, respectively.

Finally, we consider the case $l\equiv 2$ (mod 3). Suppose that $(f(v_2), f(u_2))\in \{(6, 0), (6, 2), (0, 4), (0, 6)\}$. If $f(v_1)\not=f(v_2)$, then by Lemma~\ref{le43-3-1} we have an extendable 6-$L(2, 1)$-labeling of path $u_2u_3u_4u_5$ such that $f(u_5)=f(v_2)$, and $u_6, \ldots, u_{l}$ can be labeled using pattern $f(u_1)\, f(u_2)\, f(v_2)$. If $f(v_1)=f(v_2)$, let $a\in \{0, 2, 4, 6\} \setminus \{f(u_1), f(u_2), f(v_2)\}$. Then, by Lemma~\ref{le43-3-1}, we have an extendable 6-$L(2, 1)$-labeling of path $u_2u_3u_4u_5$ such that $f(u_5)=a$, and $u_6, \ldots, u_{l}$ can be labeled using pattern $f(u_1)\, f(u_2)\, a$.

Suppose then that $(f(v_2), f(u_2))\in \{(4, 6), (2, 0)\}$. If $f(v_1)\not=f(v_2)$,  then by Lemma~\ref{le43-3-1} we have an extendable 6-$L(2, 1)$-labeling of path $u_2u_3u_4u_5$ such that $f(u_5)=f(v_2)$, and $u_6, \ldots, u_{l}$ can be labeled using pattern $f(u_1)\, f(u_2)\, f(v_2)$. Assume $f(v_1)=f(v_2)$. Let $a\in \{0, 2, 4, 6\} \setminus \{f(v_2), f(u_1), f(u_2)\}$. If $(f(v_2), f(u_2), f(u_1))\in \{(4, 6, 2), (2,0,4)\}$, then we have an extendable 6-$L(2, 1)$-labeling of path $u_2u_3u_4u_5$ such that $f(u_5)=a$ by Lemma~\ref{le43-3-1}, and $u_6, \ldots, u_{l}$ can be labeled using pattern $f(u_1)\, f(u_2)\, a$. If $(f(v_2), f(u_2), f(u_1))\in \{(4, 6, 0), (2, 0, 6)\}$, then by Lemma~\ref{le43-3-1} we have an extendable 6-$L(2, 1)$-labeling of path $u_2u_3u_4u_5$ such that $f(u_5)=f(v_2)$, and $u_6, \ldots, u_{l}$ can be labeled using pattern $f(u_1)\,f(u_2)\,a$.

\begin{table}[ht]
\caption{Partial labeling in the case when $l\equiv 2$ (mod 3)}
\centering
\begin{tabular}{c|c}
\hline
$(f(u_1), f(u_2), f(v_1),f(v_2))$ & $(f(u_3), f(u_4), f(u_5), f(u_6), \ldots, f(u_l))$\\
\hline
(4, 6, *, 2)& 1, 5, 2, 4, 6, 2, \ldots, 4, 6, 2\\
(0, 6, *, 2)& 1, 5, 2, 0, 6, 2, \ldots, 0, 6, 2\\
(6, 0, *, 4)& 3, 1, 4, 6, 0, 4, \ldots, 6, 0, 4\\
(2, 0, *, 4)& 3, 1, 4, 2, 0, 4, \ldots, 2, 0, 4\\
(0, 2, *, 4)& 5, 1, 4, 0, 2, 4, \ldots, 0, 2, 4\\
(0, 4, *, 2)& 1, 6, 2, 0, 4, 2, \ldots, 0, 4, 2\\
(6, 4, *, 2)& 1, 5, 2, 6, 4, 2, \ldots, 6, 4, 2\\
(0, 4, *, 6)& 1, 3, 6, 0, 4, 6, \ldots, 0, 4, 6\\
(2, 4, *, 6)& 1, 3, 6, 2, 4, 6, \ldots, 2, 4, 6\\
(4, 2, *, 0)& 5, 3, 0, 4, 2, 0, \ldots, 4, 2, 0\\
(6, 2, *, 0)& 5, 3, 0, 4, 2, 0, \ldots, 4, 2, 0\\
\hline
\end{tabular}
\end{table}

In the remaining case where $(f(v_2), f(u_2))\in \{(2, 6), (4, 0), (2, 4), (4, 2), (6, 4), (0, 2)\}$, we give a 6-$L(2, 1)$-labeling in Table 1 when $f(v_1)\not=f(v_2)$ with one exception that $(f(u_1), f(u_2), f(v_2))=(6, 2, 4)$. In this exceptional case, if $f(v_1)\not=1$, then we label $u_3, u_4, \ldots, u_{l-3}$ using patten $0\, 6\, 2$, and $u_{l-2}, u_{l-1}, u_l$ are labeled by 0, 3, 5, respectively; if $f(v_1)=1$, then we label $u_l, u_{l-1},\ldots, u_6$ using patten $4\,2\,6$, and $u_5, u_4, u_3$ are labeled by 4, 0, 5, respectively. In Table 1, the labels in the first column are the given labels of $u_1,u_2,v_1,v_2$, where $*$ stands for a label from either $\{0, 2, 4, 6\}$ or  $\{1, 3, 5\}$ as $|\{f(u_1, f(u_2), f(v_1), f(v_2)\} \cap \{1, 3, 5\}|\leq 1$. In the second column of Table 1, the first three labels are assigned to $u_3, u_4$ and $u_5$, respectively, and the rest labels are assigned to $u_6, \ldots, u_l$ using the shown pattern. It remains to consider $f(v_1)=f(v_2)$. If $(f(v_2), f(u_2))\in \{(2, 6), (4, 0)\}$,  let $a\in \{0, 2, 4, 6\} \setminus \{f(u_1), f(u_2), f(v_2)\}$. By Lemma~\ref{le43-3-1} we have an extendable 6-$L(2, 1)$-labeling of path $u_2u_3u_4u_5$ such that $f(u_5)=a$ and $u_6, \ldots, u_{l}$ can be labeled using pattern $f(u_1)\, f(u_2)\, a$. If $(f(v_2), f(u_2))\in \{(2, 4), (4, 2), (6, 4), (0, 2)\}$, let $a\in \{0, 2, 4, 6\} \setminus \{f(u_1), f(u_2), f(v_2)\}$. By Lemma~\ref{le43-3-1} we have an extendable 6-$L(2, 1)$-labeling of path $u_{1}u_{l}u_{l-1}u_{l-2}$ such that $f(u_{l-2})=a$ and $u_{l-3}, \ldots, u_3$ can be labeled using pattern $f(u_2)\, f(u_1)\, a$.

In each possibility above, by Lemmas~\ref{le00} and \ref{lem01}, we obtain a 6-$L(2, 1)$-labeling $f_1$ of $H$ with the desired property.
\end{proof}

\begin{thm}
\label{le43}
Any path-extendable 6-$L(2, 1)$-labeling $f$ of $P$ can be extended to a 6-$L(2, 1)$-labeling $f_1$ of $H$ such that
 $f_1|_{C-\{u_1u_2\}}$ is a path-extendable 6-$L(2,1)$-labeling of the path.
\end{thm}
\begin{proof}
By Lemma \ref{le43-3}, we may assume $|\{f(v_1), f(v_2), f(u_1), f(u_2)\}\cap \{1, 3, 5\}|\geq 2$. Assume first $|\{f(v_1), f(v_2), f(u_1), f(u_2)\}\cap \{1, 3, 5\}|=3$, so that $\{f(u_1), f(u_2)\}\cap \{1, 3, 5\}\not=\emptyset$. By Lemma \ref{le43-2}, $l\equiv 1$ or 2 (mod 3). In each case we can label $u_3, u_l$ by distinct $a, b\in\{0, 2, 4, 6\}$ respectively.  By Lemma~\ref{le43-1}, $f$ can be extended to a 6-$L(2, 1)$-labeling $f_1$
of $H$ such that
 $f_1|_{C-\{u_1u_2\}}$ is an extendable 6-$L(2,1)$-labeling. It remains to consider the case
$|\{f(v_1), f(v_2), f(u_1), f(u_2)\}\cap \{1, 3, 5\}|=2$.
We distinguish the following cases.

\medskip

\n{\bf Case 1.} $|\{f(u_1), f(u_2)\}\cap \{1, 3, 5\}|=2$.

\medskip

By Lemma \ref{le43-2}, $l\equiv 1$ or 2 (mod 3). We have $f(v_1), f(v_2)\in \{0, 2, 4, 6\}$. Note that each of $f(u_1)$ and $f(u_2)$ has two available neighbor labels in $\{0, 2, 4, 6\}$. Thus both $u_3$ and $u_l$ can be assigned labels in $\{0, 2, 4, 6\}$. We claim that $u_3$ and $u_l$ can be assigned labels in $\{0, 2, 4, 6\}$ such that $f(u_3)\not=f(u_l)$. Suppose otherwise. Then $f(u_1)$ and $f(u_2)$ have a common available neighbor label in $\{0, 2, 4, 6\}$. This implies that  $(f(u_1), f(u_2)) \in \{(3, 5), (1, 3)\}$ and $f(v_1)\not=f(v_2)$. If $(f(u_1), f(u_2))=(3,5)$, then $f(v_1)=6$ and $f(v_2)=2$, which contradicts the assumption that $f$ is an extendable $L(2,1)$-labeling of the path $v_1u_1u_2v_2$, because $(2, 5, 3, 6)$ is not an extendable $L(2,1)$-labeling of this path. Similarly, if $(f(u_1), f(u_2))=(1,3)$, then $f(v_1)=4$ and $f(v_2)=0$, which contradicts the assumption that $f$ is an extendable $L(2,1)$-labeling of the path $v_1u_1u_2v_2$. Thus $u_3$ and $u_l$ can be assigned labels in $\{0, 2, 4, 6\}$ such that $f(u_3)\not=f(u_l)$. By Lemma~\ref{le43-1}, $f$ can be extended to a 6-$L(2, 1)$-labeling $f_1$ of $H$ such that
 $f_1|_{C-\{u_1u_2\}}$ is an extendable 6-$L(2,1)$-labeling.

\medskip
\n{\bf Case 2.} $|\{f(u_1), f(u_2)\}\cap \{1, 3, 5\}|=1$.
\medskip

By Lemma \ref{le43-2}, $l\equiv 1$ or 2 (mod 3). By symmetry, we may assume $f(u_1)\in \{1, 3, 5\}$. Since
$|\{f(u_1), f(u_2), f(v_1), f(v_2)\}\cap \{1, 3, 5\}|=2$, we have $|\{f(v_1), f(v_2)\} \cap \{1, 3, 5\}|=1$.  We claim that $f(v_1)\in \{1, 3, 5\}$. Suppose otherwise. Then $f(v_2)\in \{1,3, 5\}$. In the case $f(u_1)=1$, we have $f(u_2)=6$, $f(v_2)=3$ and $f(v_1)=4$, which implies that $f$ is not an extendable 6-$L(2, 1)$-labeling of $P$, a contradiction. Thus $f(v_1)\in \{1, 3, 5\}$ and $f(v_2)\in \{0, 2, 4, 6\}$. The cases where $f(u_1)\in \{3, 5\}$ can be dealt with similarly. Since $f(u_1)$ has two available neighbor labels in $\{0, 2, 4, 6\}$, namely $f(u_2)$ and $a$, $u_l$ can be assigned $a$, and $u_3$ can be assigned $b\in \{0, 2, 4, 6\} \setminus \{f(u_2, f(v_2), a\}$. By Lemma~\ref{le43-1}, $f$ can be extended to a 6-$L(2, 1)$-labeling $f_1$ of $H$ such that $f_1|_{C-\{u_1u_2\}}$ is an extendable 6-$L(2,1)$-labeling.

\medskip
\n{\bf Case 3.} $\{f(u_1), f(u_2)\}\cap \{1, 3, 5\}=\emptyset$.
\medskip

In this case, $f(v_1), f(v_2)\in \{1, 3, 5\}$. By symmetry, we may assume $f(u_2)<f(u_1)$.  If $l\equiv 0$ (mod 3), let $a\in \{0, 2, 4, 6\} \setminus \{f(u_1), f(u_2)\}$ and label $u_3$ and $u_l$ by $a$. By Lemma~\ref{le43-1},  $f$ can be extended to a 6-$L(2, 1)$-labeling $f_1$ of $H$ such that $f_1|_{C-\{u_1u_2\}}$ is an extendable 6-$L(2,1)$-labeling. Thus, $l\equiv 1$ or 2 (mod 3).

First, we assume $l\equiv 1$ (mod 3).   If $(f(u_1), f(u_2))\in \{(6, 4), (6, 2), (6, 0), (4, 0)\}$, let $c\in \{0, 2, 4, 6\}\setminus \{f(u_1), f(u_2)\}$. We label $u_3$ by $c$ and we have an extendable 6-$L(2, 1)$-labeling of path $u_3u_4u_5u_6$ such that $f(u_6)=f(u_2)$.  If $l=7$, then label $u_7$ by $c$. Otherwise, we can label $u_{7}, \ldots, u_{l-1}$ using pattern $c\, f(u_1)\,f(u_2)$, and label  $u_l$ by $c$.   If $f(u_1), f(u_2))=(2, 0)$, then $f(v_1)=5$ and $f(v_2)\in \{3,5\}$. Thus $u_l$ and $u_{l-1}$ can be labeled 6 and 3, respectively, and $u_3, \ldots, u_{l-2}$ can be labeled using pattern 4 2 0. If $(f(u_1), f(u_2))=(4, 2)$, then $f(v_1)=1$ and $f(v_2)=5$. Thus $u_l$, $u_{l-1}$ and $u_{l-2}$ can be assigned 6, 3 and 1, respectively, $u_{l-3},\ldots, u_5$ can be labeled using pattern $4\, 0\, 2$, and $u_4$ and $u_3$ are assigned 4 and 0, respectively.

Next we assume $l\equiv 2$ (mod 3). By symmetry, we may assume $f(u_1)<f(u_2)$. Since $|\{f(v_1), f(v_2)\}\cap \{1, 3, 5\}|=2$, $f(v_1)\not=f(v_2)$. If $(f(u_1),f(u_2))\in \{(0, 2), (0, 4)\}$, then  let $b=6$. By Lemma~\ref{le43-3-1}, we have an extendable 6-$L(2, 1)$-labeling of path $u_1u_{l}u_{l-1}u_{l-2}$ such that $f(u_{l-2})=b$ and $f(u_l)\not= f(v_1)$,  while $u_{l-3}, u_{l-4}, \ldots, u_3$ are labeled using pattern $f(u_2)\, f(u_1)\, b$.
If $(f(u_1),f(u_2))=(0, 6)$, then $(f(v_1), f(v_2))\in \{(3, 1), (5, 1), (5, 3)\}$. Since $f$ is an extendable $L(2, 1)$-labeling of $P$, $(f(v_1), f(v_2))\not=(5, 3)$. In each case when $(f(v_1), f(v_2))=(3, 1)$ or (5, 1), we label $u_l, u_{l-1}, \ldots, u_6$ using patten $4\, 6\, 0$ and label $u_5, u_4, u_3$ by 4, 1, 3 respectively.
If $(f(u_1),f(u_2))\in \{(2, 6), (4, 6)\}$, then   let $b=0\in \{0, 2, 4\}\setminus \{f(u_1, u_2\}$. By Lemma~\ref{le43-3-1}, we have an extendable 6-$L(2, 1)$-labeling of path $u_2u_3u_3u_5$ such that $f(u_5)=0=b$ and $f(v_2)\not=f(u_3)$,  while $u_{6}, u_{7}, \ldots, u_l$ are labeled using pattern $f(u_1)\, f(u_2)\, b$.
It remains to consider the case when $(f(u_1),f(u_2))=(2,4)$. Then $f(v_1)=5$ and $f(v_2)=1$. We label $u_3, u_4, u_5$ by 0, 5, 1, respectively, and $u_6\ldots, u_l$ using pattern $4\, 0\, 6$.

In each possibility above, by Lemmas~\ref{le00} and \ref{lem01}, we obtain a 6-$L(2, 1)$-labeling $f_1$ of $H$ with the desired property.
\end{proof}

\section{Extension technique 2}
\label{sec:ext2}

\begin{notn}
\emph{Let $P=v_1v_2v_3$ be a path and $C=u_1u_2\ldots u_lu_1$ a cycle, $l \ge 3$, such that $V(P)\cap V(C)=\emptyset$. Throughout this section, $K$ is the graph obtained from $P$ and $C$ by adding the edge $u_2v_2$ between $P$ and $C$, and $f$ is a given 6-$L(2, 1)$-labeling of $P$.}
\end{notn}

\begin{lemma}
\label{le44-1} If $l\equiv 0$ (mod 3), then $f$ can be extended to  a 6-$L(2,1)$-labeling
$f_1$ of $K$ such that $f_1|_{C}$ is a cycle-extendable 6-$L(2,1)$-labeling of type 2 in $C$.
\end{lemma}
\begin{proof}
If $f(v_2)\in \{0, 2, 4, 6\}$, then  choose $a\in \{0, 2, 4, 6\} \setminus \{f(v_1), f(v_2), f(v_3)\}$ and assign it to $u_2$. Take $b, c\in \{0,2, 4, 6\} \setminus \{a, f(v_2)\}$. We label $u_1, u_3$ by $b, c$ respectively and $u_4, u_5, \ldots, u_{l}$ using pattern $b\, a\, c$.

Assume $f(v_2)\in \{1,3, 5\}$. Suppose first that $|\{f(v_1), f(v_3)\}\cap \{0, 2, 4, 6\}|\leq 1$.  Since $f(v_2)$ has two available neighbor labels in $\{0, 2, 4, 6\}$, we assign its other available neighbor label $a$ to $u_2$. Similarly,  Take $b, c\in \{0,2, 4, 6\} \setminus \{a, f(v_2)\}$. We label $u_1, u_3$ by $b, c$ respectively and $u_4, u_5, \ldots, u_{l}$ using pattern $b\, a\, c$. Now suppose that  $\{f(v_1), f(v_3\} \in \{0, 2, 4, 6\}$. Assign $u_2$ a label from $\{1, 3, 5\} \setminus \{f(v_2)\}$. Then $f(u_2)$ has two available neighbor labels $a, b$ in $\{0, 2, 4, 6\}$. We assign $a$ and $b$ to $u_1$ and $u_3$, respectively. Choose $c\in \{0,2, 4, 6\} \setminus \{a, b\}$ and label $u_4, u_5, \ldots, u_{l}$ using pattern $a\, c\, b$.

In each possibility above, by Lemmas~\ref{le00} and \ref{lem01}, we obtain a 6-$L(2, 1)$-labeling $f_1$ of $H$ with the desired property.
\end{proof}

\begin{lemma}
\label{le44-2} If $l\equiv 1$ (mod 3),
then $f$ can be extended to  a 6-$L(2,1)$-labeling $f_1$ of $K$ such that $f_1|_{C}$ is a cycle-extendable 6-$L(2,1)$-labeling of type 2.
 \end{lemma}
\begin{proof} We first assume $f(v_2)\in \{1, 3, 5\}$ and $|\{f(v_1), f(v_3)\}\cap \{1, 3, 5\}|\leq 1$. In this case, there exists a label in $\{1, 3, 5\}$ which can be assigned to $u_2$. There are two available neighbor labels  $a, b\in\{0, 2,4, 6\}$ for $f(u_2)$ such that  $u_1$ and $u_3$  can be assigned $a$ and $b$, respectively. Label $u_4$ by some $c\in \{0, 2, 4, 6\} \setminus \{a, b\}$ and $u_5, \ldots, u_{l}$ using pattern $a\, b\, c$.

Next assume $f(v_1), f(v_2), f(v_3)\in \{1, 3, 5\}$. If $f(v_2)=1$, then label $u_2$ by 6, $u_1, u_{l}, u_{l-1}$ by 3, 0, 2, respectively, and $u_{l-2}, \ldots, u_3$ using pattern $4\, 0\, 2$; if $f(v_2)=3$, then label $u_2$ by 0, $u_1, u_{l}, u_{l-1}$ by 5, 3, 6, respectively, and $u_{l-2}, \ldots, u_3$ using pattern $0\, 2\, 6$; if $f(v_2)=5$, then label $u_2$ by 0, $u_1, u_{l}, u_{l-1}$ by 3, 1, 6, respectively, and $u_{l-2}, \ldots, u_3$ using pattern $0\, 2\, 6$.

Finally, we assume $f(v_2)\in \{0, 2, 4, 6\}$. We label $u_2$ by some $a\in \{0, 2, 4, 6\} \setminus \{f(v_1), f(v_2), f(v_3)\}$. If $a\in \{2, 4\}$ has only one available neighbor label $d\in \{1, 3, 5\}$, then we can assign $d$ to $u_1$; if $a\in \{0,6\}$, then we choose its available neighbor label $d=3$ and assign 3 to $u_1$.  Moreover, $d$ has another available neighbor label $b$ in $\{0, 2, 4, 6\}$. Choose $c\in \{0, 2, 4, 6\} \setminus \{a, b, f(v_2)\}$. We label $u_2, u_3, \ldots, u_{l-2}$ using pattern $c\, b\, a$, and $u_{l-1}, u_l$ by $c, b$, respectively.

In each case above, by Lemmas~\ref{le00} and \ref{lem01}, we obtain a 6-$L(2, 1)$-labeling $f_1$ of $K$ with the desired property.
\end{proof}

\begin{thm}
\label{le44}
$f$ can be extended to a 6-$L(2,1)$-labeling $f_1$ of $K$ such that $f_1|_{C}$ is a cycle-extendable 6-$L(2,1)$-labeling of type 2.
\end{thm}
\begin{proof}
By Lemmas~\ref{le44-1} and \ref{le44-2}, we are left with the case $l\equiv 2$ (mod 3).  We first assume that $f(v_1), f(v_2), f(v_3)\in \{1, 3, 5\}$. If $f(v_2)=1$, then we label $u_2, u_1, u_{l}, u_{l-1}, u_{l-2}$ by $6, 3, 1, 4, 0$, respectively, and label $u_{l-3}, \ldots, u_3$ using pattern $6\, 4\, 0$.
If $f(v_2)=3$, then label $u_2, u_1, u_{l}, u_{l-1}, u_{l-2}$ by $6, 1, 5, 0, 2$, respectively, and $u_{l-3}, \ldots, u_3$ using pattern $6\, 0\, 2$. If $f(v_3)=5$, then label $u_2, u_1, u_{l}, u_{l-1}, u_{l-2}$ by $0, 3, 5, 2, 4$, respectively, and $u_{l-3}, \ldots, u_3$ using pattern $0\, 2\, 4$.

Next we assume $f(v_2)\in \{1, 3, 5\}$ and $|\{f(v_1),  f(v_3)\}\cap \{1, 3, 5\}|\leq 1$.
We assign $u_2$ a label from $\{1, 3, 5\} \setminus \{f(v_1), f(v_2), f(v_3)\}$ and then assign $u_1$ a label from $\{1, 3, 5\} \setminus \{f(u_2), f(v_2)\}$. Let $x\in \{1, 3, 5\}$. Denote by $L_x$ the set of available neighbor labels for $x$ in $\{0, 2, 4, 6\}$. Note that for each of $f(u_1)$ and $f(u_2)$, there are two available neighbor labels in $\{0, 2, 4, 6\}$. By Lemma 2.4, $|L_{f(u_1)} \setminus L_{f(u_2)}|=2$ or $1$. In the former case, we can choose an available neighbor label $a\in \{0, 2, 4, 6\}$ for $f(u_1)$ and an available neighbor label $b\in \{0, 2, 4, 6\}$ for $f(u_2)$ such that $a\not=b$. Choose $c\in \{0, 2, 4, 6\} \setminus \{a, b, d\}$, where $L_{f(u_1)}=\{a, d\}$. In the latter case, let $b\in L_{f(u_1)} \setminus L_{f(u_2)}$, $a\in L_{f(u_1)}\setminus \{b\}$ and $c\in \{0, 2, 4, 6\} \setminus \{a, b, d\}$, where $L_{f(u_1)}=\{b, d\}$. In both cases we label $u_3, u_4, \ldots, u_l$ using pattern $b\, c\, a$.

Finally, we assume $f(v_2)\in \{0,2, 4, 6\}$. We assign $u_2$ a label $a\in \{0, 2, 4, 6\} \setminus \{f(v_1), f(v_2), f(v_3)\}$. Then we assign $u_1, u_l$ labels $d_1, d_2$ from $\{1, 3, 5\}$, respectively, such that $a$ is not an available neighbor label of $d_2$. Let $b$ be an available neighbor labels in $\{0, 2, 4, 6\}$ for $d_2$. Choose $c\in \{0, 2, 4, 6\} \setminus \{a, b, f(v_2)\}$. We label $u_3, u_4$ by $c, b$, respectively, and $u_5 \ldots, u_{l-1}$ using pattern $a\, c\, b$.

In each case above, by Lemmas~\ref{le00} and \ref{lem01}, we obtain a 6-$L(2, 1)$-labeling $f_1$ of $K$ with the desired property.
\end{proof}

\section{Proof of Theorem \ref{thm1}}
\label{sec:ubound}

Throughout this section $G$ is an outer plane graph with $\Delta=3$. A path $P = v_1v_2\ldots v_t$ of $G$ is called a {\em branch} if $d(v_1)\geq 3$, $d(v_t)\geq 3$ and $d(v_i)=2$ for $2 \le i \le t-1$. For two blocks $A$ and $B$ of $G$, define $d(A, B)=\min \{d(x, y): x\in V(A), y\in V(B)\}$, where $d(x, y)$ is the distance in $G$ between $x$ and $y$.
Let $B_1$ and $B_2$ be two blocks of $G$ such that $d(B_1, B_2)$ is minimized. Since $\Delta=3$, $B_1$ is joined to $B_2$ by a branch of length at least one.

\medskip

\n{\bf Proof of Theorem~\ref{thm1}}. Suppose to the contrary that not every outerplanar graph with maximum degree 3 satisfies $\lambda \le 6$. Let $G$ be a smallest counterexample. That is, $G$ is an outerplanar graph of maximum degree 3 having no 6-$L(2,1)$-labelings such that $|V(G)|$ is minimum. Clearly, $|V(G)|\geq 4$ and $G$ is connected by the minimality of $G$. We prove the following claim first.

\medskip
\n{\bf Claim.} $G$ is 2-connected.
\medskip

\n{\em Proof of the Claim.}
Suppose $G$ is not 2-connected. Since $G$ is connected, it has a cut edge. By Lemma~\ref{le42}, $G$ has no vertex of degree 1. Thus $G$ consists of blocks and branches connecting blocks. We construct a graph $X$ as follows: $V(X)$ is the set of blocks of $G$; for $x, y\in V(X)$, let $B_x$ and $B_y$ denote the blocks of $G$ corresponding to $x$ and $y$, respectively.  Vertex $x$ is adjacent to vertex $y$ in $X$ if and only if block $B_x$ is connected to block $B_y$ by a branch of $G$. It is obvious that $X$ is a tree. Let $u$ be a vertex of $X$ with degree one and $v$ the unique neighbor of $u$ in $X$. Let $P=v_1v_2\ldots v_k$ be the branch connecting $B_u$ and $B_v$, where $v_1\in V(B_v)$ and $v_k\in V(B_u)$. Then $k\geq 2$ as $\Delta=3$. Let $G_1$ denote the graph obtained from $G-V(B_u)$ by deleting $v_2, \ldots, v_{k-1}$. Since $|V(B_u)|\geq 3$, we have $|V(G_1)|\leq |V(G)|-2$. By the choice of $G$, $G_1$ has a 6-$L(2, 1)$-labeling. By Lemma~\ref{le42}, $G-V(B_u)$ has a 6-$L(2, 1)$-labeling.

If $k=2$, let $u_1$ and $u_2$ be two neighbors of $v_1$ in $B_v$ and let $B_u'$ be the graph induced by $V(G_u)\cup \{v_1, u_1, u_2\}$. Note that $u_1v_1u_2$ has a 6-$L(2,1)$-labeling which, by Theorems~\ref{le44} and \ref{le43}, can be extended to a 6-$L(2,1)$-labeling of $B_u'$. Thus $G$ has a 6-$L(2,1)$-labeling, contradicting our assumption.

Thus we assume $k\geq 3$. Define $G_2=G-(V(G_1)\cup \{v_2,\ldots v_{k-3}\})$. That is, $G_2$ is obtained from the block $B_u$ by adding the path $v_kv_{k-1}v_{k-2}$.  Note that $v_{k-1}$ and $v_{k-2}$ have been assigned labels from [0, 6]. To prove Theorem~\ref{thm1}, it is sufficient to prove that the existing 6-$L(2, 1)$-labeling of $v_{k-2}v_{k-1}$ can be extended to a 6-$L(2, 1)$-labeling of $G_2$. To apply Theorem~\ref{le44}, we construct a graph $G^*_2$ obtained from $G_2$ by adding to $G_2$ a new vertex $w$ together with an edge joining $w$ and $v_{k-1}$.  Now we label $w$ as follows: if $f(v_{k-1})\in \{0, 2, 4, 6\}$, then $w$ is assigned a label from $\{0, 2, 4, 6\}\setminus \{f(v_{k-2}), f(v_{k-1})\}$; if $f(v_{k-1})\in \{1, 3, 5\}$, then $w$ is assigned a label from $\{1, 3, 5\}\setminus \{f(v_{k-2}), f(v_{k-1})\}$. Clearly, our labeling of $wv_{k-1}v_{k-2}$ is a 6-$L(2,1)$-labeling. By Theorems~\ref{le44} and \ref{le43}, the 6-$L(2,1)$-labeling of $wv_{k-1}v_{k-2}$ can be extended to a 6-$L(2,1)$-labeling $f$ of $G^*_2$. Clearly, $f|_{G_2}$ is a 6-$L(2, 1)$-labeling of $G_2$, which together with a 6-$L(2,1)$-labeling of $G-V(B_u)$ gives a 6-$L(2,1)$-labeling of $G$, a contraction. Therefore, $G$ is 2-edge-connected. Since $G$ is an outer plane graph with $\Delta(G)=3$, $G$ must be 2-connected.
$\square$

\medskip
By the Claim, $G$ is 2-connected. If $|V(G)|=4$, then $G$ is isomorphic to the complete $K_4$ with one edge removed, and so $G$ has a 6-$L(2,1)$-labeling. Assume $|V(G)|\geq 5$. Since $\Delta=3$, $G$ contains at least two vertices of degree 3. It follows that $G$ contains two adjacent faces $F_1$ and $F_2$. If both $F_1$ and $F_2$ are two 3-faces, then by Lemma~\ref{l1}, $G$ is isomorphic to the complete $K_4$ with one edge removed, contracting $|V(G)|\geq 5$. Thus $G$ contains a face $F$ with $|\partial F|\geq 4$. Denote $\partial F=v_1v_2\ldots v_l$. We assign labels from $[0, 6]$ to the vertices of $\partial F$ in the follow way.

If $l\equiv 0$ (mod 3), then label $v_1, v_2, \ldots, v_l$ using pattern 0 2 4;  if $l\equiv 1$ (mod 3), then label $v_1$ by 3 and $v_2, \ldots, v_l$ using pattern $6\, 4\, 0$; if $l\equiv 2$ (mod 3), then label $v_1, v_2$ by 3, 1, respectively, and $v_3, \ldots, v_l$ using pattern $4\, 2\, 0$. By Lemmas~\ref{le00} and \ref{lem01}, the
labeling of $\partial F$ is a cycle-extendable 6-$L(2, 1)$-labeling of type 1.  If $G=F$, we are done. Assume $G\not=F$. Let $F'$ be a face adjacent to $F$. Since $G$ is an outer plane graph and $G$ is 2-connected, $|E(F)\cap E(F')|=1$. We assume, without loss of generality, that $E(F)\cap E(F')=v_2v_3$.  By Theorem~\ref{le43}, the 6-$L(2, 1)$-labeling of $v_1v_2v_3v_4$ can be extended to a path-extendable 6-$L(2,1)$-labeling of $\partial F'-v_2v_3$. Repeating the extendable procedure above, finally we obtain that $G$ has a 6-$L(2,1)$-labeling. This contradiction proves the upper bound in Theorem~\ref{thm1}.

In particular, we have $\lambda(G(l)) \le 6$. This and Theorem \ref{lower} imply $\lambda(G(l)) = 6$ when $l \ge 4$ is not a multiple of $3$.
$\blacksquare$

\end{document}